\documentclass[11pt,a4paper]{amsart}
\usepackage{amssymb}
\usepackage{amscd}
\usepackage{amsmath,amsfonts,stmaryrd}
\input xy 
\xyoption{all}

\theoremstyle{plain}
\newtheorem{thm}{Theorem}[section]
\newtheorem{lem}[thm]{Lemma}
\newtheorem{lemdef}[thm]{Lemma/Definition}
\newtheorem{pro}[thm]{Proposition}

\theoremstyle{definition}
\newtheorem{dfn}[thm]{Definition}

\newtheorem{rem}[thm]{Remark}

\DeclareMathOperator{\Hom}{Hom}

\DeclareMathOperator{\End}{End}
\DeclareMathOperator{\Ext}{Ext}
\DeclareMathOperator{\EXT}{Ext}
\DeclareMathOperator{\Rep}{R}
\DeclareMathOperator{\RRep}{\mathbf{R}}

\DeclareMathOperator{\Spec}{Spec}

\DeclareMathOperator{\modu}{mod}
\DeclareMathOperator{\Modu}{Mod}

\newcommand{\EKmod}{ \widehat{\mathcal{E}}}

\newcommand{\PKmod}{\widehat{\mathcal{P}} }
\newcommand{\QKmod}{ \widehat{\mathcal{Q}}}
\newcommand{\HKmod}{\widehat{\mathcal{H}}  }
\newcommand{\QPKmod}{\widehat{\mathcal{Q}/\mathcal{P}}  }
\newcommand{\HPKmod}{\widehat{\mathcal{H}/\mathcal{P}}  }

\newcommand{\RKmod}{ \widehat{\mathcal{R}}}
\newcommand{\SKmod}{ \widehat{\mathcal{S}}}
\newcommand{\RSKmod}{\widehat{\mathcal{R}/\mathcal{S}}  }

\newcommand{\eprime}{e}

\newcommand{\rp}{r}
\newcommand{\cp}{c}

\newcommand{\ellq}{\ell^\prime}
\newcommand{\rQ}{r^\prime}

\newcommand{\Gl}{\mathbf{Gl}}

\DeclareMathOperator{\id}{id}
\DeclareMathOperator{\idim}{inj.dim}
\DeclareMathOperator{\pdim}{proj.dim}

\DeclareMathOperator{\Dim}{\rm \underline{dim}}

\newcommand{\oTo}{\xymatrix{ \ar@{^{(}->}[r]|{\mathbf{O}}& }} 
\newcommand{\cTo}{\xymatrix{ \ar@{^{(}->}[r]|{\mathbf{|}}& }} 
\newcommand{\coTo}{\xymatrix{ \ar@{^{(}->}[r]|{\mathbf{O}}|{\mathbf{|}}& }} 
\DeclareMathOperator{\surj}{\twoheadrightarrow }

\DeclareMathOperator{\Ker}{ker}
\DeclareMathOperator{\coKer}{coker}
\DeclareMathOperator{\Bild}{Im}

\newcommand{\ep}{\varepsilon}

\newcommand{\la}{\lambda}

\newcommand{\al}{\alpha}

\newcommand{\N}{\mathbb{N}}

\newcommand{\R}{\mathbb{R}}

\newcommand{\Z}{\mathbb{Z}}

\newcommand{\n}{\mathfrak{n}}
\newcommand{\m}{\mathfrak{m}}

\newcommand{\mcA}{\mathcal{A}}
\newcommand{\mcB}{\mathcal{B}}
\newcommand{\mcC}{\mathcal{C}}

\newcommand{\mcE}{\mathcal{E}}
\newcommand{\mcF}{\mathcal{F}}

\newcommand{\mcH}{\mathcal{H}}

\newcommand{\mcO}{\mathcal{O}}
\newcommand{\mcP}{\mathcal{P}}
\newcommand{\mcQ}{\mathcal{Q}}
\newcommand{\mcR}{\mathcal{R}}
\newcommand{\mcS}{\mathcal{S}}
\newcommand{\mcT}{\mathcal{T}}

\newcommand{\mcX}{\mathcal{X}}
\newcommand{\mcY}{\mathcal{Y}}
\newcommand{\mcZ}{\mathcal{Z}}

\DeclareMathOperator{\Gen}{Gen}
\DeclareMathOperator{\Cogen}{Cogen}

\DeclareMathOperator{\Add}{Add}
\DeclareMathOperator{\Dual}{D}
\DeclareMathOperator{\Gr}{Gr}

\newcommand{\GrMd}{\Gr_A \binom{M}{\underline{d}}}

\newcommand{\GrcMd}{\Gr_B \binom{\cp (M)}{\underline{d}, \underline{r}}}

\newcommand{\Grcdi}{\Gr_B \binom{\cp (M)}{\underline{d}, \underline{r}_i}}

\newcommand{\RdA}{\Rep_A(\underline{d})}
\newcommand{\RdrB}{\Rep_B(\underline{d},\underline{r})}

\newcommand{\RdrQB}{\Rep_{KQ_B}(\underline{d},\underline{r})}
\newcommand{\RdQA}{\Rep_{KQ_A}(\underline{d})}
\newcommand{\RRdA}{\RRep_A(\underline{d})}
\newcommand{\RrQinfty}{\Rep_{KQ^{\infty}}(1,\underline{r})}
\newcommand{\RrQn}{\Rep_{KQ^{>n}}(\underline{r})}

\newcommand{\OM}{\mathcal{O}_M}

\newcommand{\SN}{\mathcal{S}_{[N]}}
\newcommand{\SwN}{\mathcal{S}_{[\cp (N)]}}

\newcommand{\dd}{\underline{d}}

\newcommand{\rr}{\underline{r}}

\newcommand{\wdd}{(\underline{d}, \underline{r})}

\newcommand{\Gd}{\mathbf{Gl}_{\underline{d }}}
\newcommand{\Glr}{\mathbf{Gl}_{\underline{r}}}

\DeclareMathOperator{\dbslash}{/\!\!/}

\begin{document}

\title{On quiver Grassmannians and orbit closures for representation-finite algebras}

\author{William Crawley-Boevey}
\address{Department of Pure Mathematics, University of Leeds, Leeds LS2 9JT, UK}
\email{w.crawley-boevey@leeds.ac.uk}

\author{Julia Sauter}
\address{Fakult\"at f\"ur Mathematik, Universit\"at Bielefeld, Postfach 100 131, D-33501 Bielefeld, Germany}
\email{jsauter@math.uni-bielefeld.de}






\begin{abstract}
We show that Auslander algebras have a unique tilting and cotilting module which is generated and cogenerated by a projective-injective; its endomorphism ring is called the projective quotient algebra.  
For any representation-finite algebra, we use the projective quotient algebra to construct desingularizations of quiver Grassmannians, orbit closures in representation varieties, and their desingularizations. This generalizes results of Cerulli Irelli, Feigin and Reineke. 
\end{abstract}
\maketitle

\section{Introduction}
Starting from a finite-dimensional module $M$ over a finite-dimensional associative algebra $A$ one can obtain an algebraic variety either by forming the so-called quiver Grassmannian 
$\GrMd$ parameterizing submodules of $M$ of dimension vector $\dd$, or by taking the orbit closure 
$\overline{\mcO_M}$ associated to $M$ in the representation variety parameterizing all $A$-modules of the 
same dimension vector as $M$.
In the situation where $A$ is the path algebra of a Dynkin quiver, Cerulli Irelli, Feigin and Reineke 
constructed desingularizations of quiver Grassmannians \cite{CIFR} and realized orbit closures 
as affine quotient varieties \cite{CFR}. One would like to generalize these constructions
to other algebras, and Keller and Scherotzke \cite{KS2} and Scherotzke \cite{Sch} obtained some results in the cases when $A$ is an iterated tilted algebra of Dynkin type or a self-injective algebra of finite representation type.
In this paper we generalize to the case when $A$ is an arbitrary finite-dimensional algebra of finite representation type.
We again construct desingularizations of quiver Grassmannians and realize orbit closures as affine quotient varieties. 
For the latter, our construction unifies the work of Cerulli Irelli, Feigin and Reineke
with a construction of closures of conjugacy classes used by Kraft and Procesi \cite{KP} to prove their normality.
In addition, we construct desingularizations of orbit closures.

In order to study a Dynkin quiver $Q$, a certain algebra (denoted $\tilde\Lambda_Q$ or $B_Q$) has been introduced independently by Hernandez and Leclerc \cite{HL} and by Cerulli Irelli, Feigin and Reineke.
We generalize this to an algebra $B=B_A$, which we call the \emph{projective quotient algebra}, associated to any finite-dimensional algebra $A$ of finite representation type.
We begin by discussing this algebra.
We write $A$-$\Modu$ for the category of (left) $A$-modules and $A$-$\modu$ for
the category of finite-dimensional modules.
Let $\mcQ$ be the category whose objects are surjective morphisms $f\colon P\to X$
in $A$-$\modu$ with $P$ projective, and whose morphisms 
from $f\colon P\to X$ to $f' \colon P' \to X'$ are given by pairs 
$(t,s)\in \Hom_A(P,P') \times \Hom_A(X, X' )$ such that $f'  t = s  f$.  
We can consider $\mcQ$ as a category of complexes with two terms,
and define $\mcH$ to be the corresponding homotopy category.
Assuming that $A$ has finite representation type, the category $\mcH$ has
only finitely many indecomposable objects, and we set $G$ to be the direct sum of them. Thus
$G$ is an additive generator for $\mcH$. The projective quotient algebra for $A$ 
is defined to be $B = \End_{\mcH}(G)^{op}$.

There is another characterization of the projective quotient algebra which shows how natural it is.
Recall that if $A$ has finite representation type, its Auslander algebra is $\End_A(E)^{op}$,
where $E$ is the direct sum of one copy of each indecomposable $A$-module.
Auslander \cite{AusII} characterized the algebras which arise this way as those 
of global dimension $\le 2$ and dominant dimension $\ge 2$.
In the same spirit we have the following. 
(All tilting and cotilting modules are assumed
to be classical, so with projective dimension $\le 1$ and 
injective dimension $\le 1$ respectively.)

\begin{lem}
\label{l:introlem}
Let $\Gamma$ be a finite-dimensional algebra
and let $\mcC$ be the class of modules generated and cogenerated by a 
projective-injective $\Gamma$-module. Then
$\Gamma$ is an Auslander algebra if and only if it has global
dimension $\le 2$ and
$\mcC$ contains a tilting module.
In this case there is a unique basic tilting module $T \in \mcC$ and it is also a cotilting module.
\end{lem}

\begin{proof}
Suppose $\Gamma$ has global dimension $\le 2$.
If $Y\in\mcC$ then there are exact sequences $0\to X\to I\to Y\to 0$ and $0\to Y\to P\to Z\to 0$
with $I$ injective and $P$ projective (in fact both projective-injective), 
and from the resulting long exact sequences it is easy to see that
$\pdim Y\le 1$, $\idim Y \le 1$ and $\Ext^1(Y,Y)=0$.
Thus any module $Y$ in $\mcC$ is a partial tilting and cotilting module, so has $\le n$ non-isomorphic
indecomposable summands, where $n$ is the number of simple $\Gamma$ modules,
and it has exactly $n$ if and only if $Y$ is tilting or equivalently cotilting.

If $\Gamma$ is an Auslander algebra, then the dominant dimension condition gives
a map $\theta:Q'\to Q''$ between projective-injective modules with $\Ker \theta \cong \Gamma$.
Then $Y = Q'\oplus \Bild\theta \in \mcC$ is a tilting module.

Conversely, if there is a tilting module $Y$ in $\mcC$ then $\Gamma$ embeds in a direct sum of copies of $Y$,
so in a projective-injective module. Thus there is an exact sequence $0\to \Gamma\to Q\to U\to 0$ with
$Q$ projective-injective. Then it is easy to see that $\pdim U \le 1$ and applying $\Hom(U,-)$ to
the first exact sequence for $Y$ above gives $\Ext^1(U,Y)=0$.
Thus, since $Y$ is a cotilting module, $U$ is cogenerated by $Y$, so it too embeds in a projective-injective.
It follows that $\Gamma$ has dominant dimension $\ge 2$, so is an Auslander algebra.
\end{proof}

Now let $A$ be an algebra of finite representation type. 
By the lemma,
its Auslander algebra $\Gamma$ has a uniquely determined basic tilting and cotilting module $T$ which is generated and cogenerated by a 
projective-injective module.

\begin{thm}
The projective quotient algebra $B$ for $A$ is isomorphic to $\End_\Gamma(T)^{op}$.
\end{thm}

This theorem and the next are are proved in \S\ref{s:pqa}.
For simplicity we suppose that $A$ is basic. 
Thus $A$ is a direct summand of $E$, 
and also $A\to 0$ is a direct summand of $G$ in the category $\mcH$, 
and we denote
by $e\in B$ the projection onto this summand. Then 
$eBe \cong \End_{\mcH}(A\to 0)^{op} \cong A$, and we treat this as an identification.
There is a recollement
\[ 
\xymatrix{
\text{$B/BeB$-$\Modu$} \ar[rr]|{i} \;  && \; \text{$B$-$\Modu$} \; \ar@<2ex>[ll]^{p} \ar@<-2ex>[ll]_{q} \ar[rr]|>>>>>>>>>{e}  && \; 
\text{$A$-$\Modu$}
\ar@<2ex>[ll]^{r } \ar@<-2ex>[ll]_{\ell }.
}
\]
where $e$ also denotes the functor sending a $B$-module $M$ to the $A$-module $eM$, and
there is an associated intermediate extension functor~$c$.
We define $\Dual(-) = \Hom_K(-,K)$.

\begin{thm}
The intermediate extension functor $c:\text{$A$-$\Modu$} \to \text{$B$-$\Modu$}$
is given by $c(X) = (\Dual T)\otimes_\Gamma \Hom_A(E,X)$.
It satisfies $\Ext^1_B(c(X),c(Y)) = 0$, $\pdim c(X)\leq 1, \idim c(X)\leq 1$ for any $A$-modules $X,Y$.
\end{thm}

We use this result together with standard properties of intermediate extension functors to obtain our geometric results.
Henceforth let $K$ be an algebraically closed field. 
Still $A$ is a basic algebra of finite representation type, and
let $e_1,\dots,e_n$ be a
complete set of primitive orthogonal idempotents in $A$.
The dimension vector of an $A$-module $N$ is $\dd = (\dim e_i N) \in \N_0^n$.
Let $e_1,\dots,e_m$ be a complete set of orthogonal idempotents in $B$, with the first 
$n\le m$ being the corresponding idempotents for $A$. Dimension vectors for $B$ are given by pairs
$(\dd,\rr)$ with $\dd\in\N_0^n$ and $\rr\in\N_0^{m-n}$.

Let $M$ be an $A$-module.
Recall that $\GrMd$ is a projective variety,
possibly with singularities.  
If $N$ is an $A$-module of dimension $\dd$, then there is a map of varieties
from the set of injective maps in $\Hom_A(N,M)$ to the Grassmannian, whose
image is the set of submodules isomorphic to $N$. We denote it $\mcS_{[N]}$;
it is locally closed and if non-empty it is irreducible.
Note that $\GrMd$ may have several irreducible components, 
but using that $A$ has finite representation type, they are of the form
$\overline{\mcS_{[N_i]}}$ for some $A$-modules $N_1,\dots,N_\ell$.
We consider $\mcS_{[c(N_i)]} \subseteq \Grcdi$, where $(\dd,\rr_i)$ is
the dimension vector of $c(N_i)$.
A \emph{desingularization} of a variety $X$ is a birational projective morphism 
from a smooth variety to $X$.
Our first geometric application is the following.

\begin{thm} \label{GrassDes} 
The functor $e$ induces a desingularization
$
\bigsqcup_{i=1}^\ell
\overline{\mcS_{[c(N_i)]}}
\to \GrMd
$.
\end{thm}

Modules of any given dimension vector $\dd\in \N_0^n$ are parameterized by a variety $\RdA$
and the group $\Gd$ acts with the orbits being the isomorphism classes. 
We write $\OM$ for the orbit corresponding to a module $M$.
Our second geometric application is as follows.

\begin{thm}
If $K$ has characteristic zero and $M$ is an $A$-module, then
$\overline{\OM}$ is isomorphic to the affine quotient variety
$\RdrB\dbslash\Glr$ where $(\dd,\rr)$ is the dimension vector of $c(M)$,
and also to $\overline{\mcO_{c(M)}} \dbslash \Glr$.
\end{thm}

Associated to the recollement there is a stability notion for $B$-modules.
We denote by $\RdrB^s$ the open subset of stable modules in $\RdrB$
and by $\RdrB^s / \Glr$ the corresponding geometric quotient (the GIT
moduli space of stable $B$-module). 
Using this we obtain another desingularization.

\begin{thm}
If $K$ has characteristic zero and $M$ is an $A$-module, then
the natural map 
\[\left( \overline{\mcO_{c(M)}} \cap \RdrB^s\right)/ \Glr \to \overline{\mcO_M}
\]
is a desingularization.
\end{thm}

Part of this work was completed while the first author was visiting CRC 701 at Universit\"at Bielefeld. The author thanks his hosts for their hospitality.

\section{Recollements}
Recall (see \cite{BBD}) that a \emph{recollement} of abelian categories is a diagram 
\[
\xymatrix{
\mcA \ar[rr]|{i} \;  && \; \mcB \; \ar@<2ex>[ll]^{p} \ar@<-2ex>[ll]_{q } \ar[rr]|{\eprime}  && \; \mcC   \ar@<2ex>[ll]^{r} \ar@<-2ex>[ll]_{\ell}
}
\]
consisting of of three abelian categories and six functors satisfying the following conditions.
\begin{itemize}
\item[(1)]
$(q,i), (i,p), (\ell , \eprime ) , (\eprime , r)$ are adjoint pairs of functors. 
\item[(2)]
The natural transformations $\rho \colon \eprime r \to \id_{\mcC}, 
\la \colon \id_{\mcC} \to \eprime \ell $ are isomorphisms.  
\item[(3)] The natural transformations $P\colon \id_{\mcA} \to pi$ and 
$Q \colon qi\to \id_{\mcA}$ are isomorphisms. 
\item[(4)] The functor $i$ is an embedding onto the full subcategory of $\mcB$ with objects $b$ such that $\eprime b =0$.  
\end{itemize}

The condition (1) implies that $i$ and $e$ are exact,and the conditions (2) and (3) are equivalent to the fully faithfulness of $\ell$, $r$ and $i$.

\subsection*{The intermediate extension functor} 
Associated to a recollement there is a functor 
$c \colon \mcC \to \mcB$
called the \emph{intermediate extension functor} 
given by 
\[
c(M) =\Bild (\ell (M)\xrightarrow{\gamma_M} r(M))
\]
where $\gamma_M$ is the adjoint of the 
inverse of the natural map $\rho_M\colon \eprime r(M)\to M$,
or equivalently the adjoint of the inverse of the natural map $\lambda_M\colon  M\to \eprime \ell (M)$. 
For $N$ in $\mcB$ set $\eprime (N):= M$. Then there are two adjoint maps 
$\al_N\colon \ell (M)\to N, \;\beta_N\colon N\to r(M)$ and the following diagram commutes 
\[ 
\xymatrix{
\ell (M)\ar[rd]_{\gamma_M} \ar[r]^{\al_N}& N\ar[d]^{\beta_N} \\
& r(M)
}
\]

The following lemma summarizes results in \cite[\S 4]{Ku} and \cite[Proposition 4.4]{FP}.

\begin{lem}\
\label{IntExt} 
\begin{itemize}

\item[(1)] $\eprime c(M)\cong M$ naturally in $M$. 
\item[(2)] Let $N$ be in $\mcB$ with $\eprime (N)\cong M$. Then $c(M)$ is a subquotient of $N$. More precisely, from the above diagram we get short exact sequences 
\[ 
\begin{aligned}
0\to \Ker \beta_N  &\to N \to \Bild \beta_N\to 0\\
0\to c(M) & \xrightarrow{j} \Bild \beta_N \to \coKer (j)\to 0 
\end{aligned}
\] 
with $\eprime (\Ker \beta_N)=0 = \eprime (\coKer (j))$. 
\item[(3)] Let $N$ be in $\mcB$ with $\eprime (N)\cong M$. Then, $N\cong c(M)$ if and only if $N$ has no non-zero subobjects or quotients in $\mcA$.
\item[(4)] The functor $c$ preserves epimorphisms and monomorphisms. 
\item[(5)] If $\mcB$ is a category with sums and products in which the natural map 
$\bigoplus_{\al} N_{\al} \to \prod_{\al} N_{\al}$ is monic for all indexed sets of objects 
$\{ N_{\al}\}$, then $c$ commutes with direct sums. 
\item[(6)] The functor $c$ maps simples in $\mcC$ to simples in $\mcB$. There is a bijection between sets of isomorphisms classes of simples 
\[ 
\{ \text{ simples in }\mcA\} \sqcup \{ \text{ simples in }\mcC\} \to \{ \text{ simples in }\mcB\}
\]
giving by sending a simple $N$ in $\mcA$ to $i(N)$ and a simple $M$ in $\mcC$ to $c(M)$. 
\item[(7)] There are short exact sequences of natural transformations 
\[ 
0\to ip\ell \to \ell \to c \to 0, \quad 0\to c\to r \to iqr \to 0
\]
\end{itemize}
\end{lem}

The next lemma is essentially in \cite{FP}.

\begin{lem}
The application of the functor $\eprime $ on the following morphisms spaces 
\begin{itemize}
\item[(1)] $\Hom_{\mcB} (c(M), F) \to \Hom_{\mcC} (\eprime c(M), \eprime (F)) \cong \Hom_{\mcC} (M, \eprime (F))$ and  
\item[(2)] $\Hom_{\mcB}  (F,c(M)) \to \Hom_{\mcC} (\eprime (F), \eprime c(M)) \cong \Hom_{\mcC} (\eprime (F), M)$
\end{itemize}
are injective. In particular, the functor $c$ is fully faithful and preserves indecomposable objects. 
\end{lem}

\begin{proof}
Since $c(M)= \Bild (\ell (M) \to r(M))$ we have a factorization  
$ \ell (M) \xrightarrow{p_M} c(M) \xrightarrow{i_M} r(M) $
with $p_M$ epimorphism and $i_M$ monomorphism. Now, 
we just need to see that we have a commutative diagram 
\[
\xymatrix{ 
\Hom_{\mcB}(c(M) , F) \ar[r]^{\eprime }\ar@{=}[d] & \Hom_{\mcC}(\eprime c(M), \eprime (F))\ar[r] & \Hom_{\mcC}(M, \eprime (F))\ar@{=}[d]\\
           \Hom_{\mcB}(c(M) , F) \ar[r]^{- \circ p_M} & \Hom_{\mcB}(\ell (M) , F) \ar[r]^{\eprime}    & \Hom_{\mcC}(M, \eprime (F))
}
\]
where the $\eprime$ in the lower row is the adjunction isomorphism. 
Since $p_M$ is an epimorphism, the lower row is an injective linear map.
The diagram commutes since $e(p_M)$ is the identity. 
Analoguously, in (2) the morphism identifies with 
$
\Hom_{\mcB}(F,c(M) ) \xrightarrow{i_M\circ -}  \Hom_{\mcB}(F,r(M) ) \xrightarrow{\eprime}   \Hom_{\mcC}(M, \eprime (F)) 
$
where $\eprime $ stands for the adjunction isomorphism. Since $i_M$ is a monomorphism, the lower line is an injective linear map. 
For the fully faithfulness, the following map is the identity   
\[ 
 \Hom_{\mcC}(M,N) \to \Hom_{\mcB}(c(M), c(N)) \to  \Hom_{\mcC}(\eprime c(M), \eprime c (N)) \cong  \Hom_{\mcC}(M,N), 
\]
therefore $ \Hom_{\mcC}(M,N)\to \Hom_{\mcB}(c(M), c(N))$ is injective and the second map $\Hom_{\mcB}(c(M), c(N)) \to  \Hom_{\mcC}(M,N)$ is surjective. 
Now, the second map is injective by the previous part of the proof, therefore the second map is an isomorphism. It follows that also $ \Hom_{\mcC}(M,N)\to \Hom_{\mcB}(c(M),c(N))$ is an isomorphism. Since $c$ is fully faithful, it preserves indecomposables.
\end{proof}

\subsection*{Stable and costable objects} Associated to a recollement there are the notions of stable and costable objects
(see \cite[\S 2.6, \S 4.8]{KS} for the case of Kan extensions). 
We consider a recollement as above 
with intermediate extension functor $c$.
We characterize stable and costable objects in the following two lemmas. For an additive functor $f\colon \mcS\to\mcT $ we denote by $\Ker f$ the full subcategory in $\mcS$  whose objects are send to zero under~$f$. 

\begin{lemdef} \label{stable}
We say that an object $F$ in the middle category $\mcB$ is \emph{stable} (or \emph{$\eprime $-stable}) 
if one of the following equivalent conditions is fulfilled
\begin{itemize}
\item[(1)] $\Hom_{\mcB} (G, F) =0$ for every object $G$ with $\eprime (G)=0$. 
\item[(2)] The natural map $F\to r\eprime (F)$ is a monomorphism.
\item[(3)] Every non-zero subobject $F^\prime \subset F$ fulfills $\eprime (F^\prime )\neq 0$. 
\item[(4)] $F\in \Ker p$. 
\end{itemize}
From (3) it follows that every subobject of a stable object is stable. 
Furthermore, if $F$ is stable, then there is a (natural) monomorphism $ c \eprime (F)\to F$ (compare Lemma~\ref{IntExt} part (2)). 
\end{lemdef}

\begin{proof}
Assume $F$ fulfills (1). Let $F^\prime \subset F$ be a non-zero subobject (if it exists). By (1) it follows that $e(F^\prime )\neq 0$, so (3) is fulfilled.
Assume $F$ fulfills (3). Let $F^\prime \subset F$ be the kernel of the map $F\to r\eprime (F)$. Since $\eprime $ is exact, we get $e(F^\prime)=0$, so by (3) it follows $F^\prime =0$. Therefore, (2) is fulfilled. 
Assume $F$ fulfills (2). Let $G$ be an object with $e(G)=0$. Using (2) we get a monomorphism 
\[
0\to \Hom_{\mcB} (G,F) \to \Hom_{\mcB} (G, re(F)) 
\]
and by the adjointness $\Hom_{\mcB} (G, re(F)) = \Hom_{\mcC} (e(G), e(F)) =0$. Therefore $\Hom_{\mcB} (G,F)=0$ and (1) is fulfilled. 
Now, (4) implies (2) by \cite[Proposition 4.9]{FP}. On the other hand, from loc.\ cit.\ Proposition 4.2 we know that the kernel of $F\to r\eprime (F)$ is precisely $ip (F)$. So, if $F \to r\eprime (F)$ is a monomorphism, then $ip (F)=0$ and since $i$ is fully faithful, it follows $p(F)=0$. 
\end{proof}

\begin{lemdef} \label{costable}
We say that an object $H$ in the middle category $\mcB$ is \emph{costable} (or \emph{$\eprime $-costable}) if one of the following equivalent conditions is fulfilled 
\begin{itemize}
\item[(1)]
$\Hom_{\mcB} (H,G) =0$ for every $G$ with $\eprime (G)=0$.  
\item[(2)] The natural map $\ell \eprime (H) \to H$ is an epimorphism. 
\item[(3)] Every non-zero quotient object $H\surj H^\prime$ fulfills $\eprime (H^\prime) \neq 0$.
\item[(4)] $H\in \Ker q$.
\end{itemize} 
From (3) it follows that every quotient of a costable object is costable again. 
Furthermore, if $H$ is costable, then there is a (natural) epimorphism $H\to c\eprime  (H)$ (compare Lemma~\ref{IntExt} part (2)). 
\end{lemdef}

\begin{proof}
Assume $H$ fulfills (1). Let $H^\prime$ be a non-zero quotient of $H$. By (1) we get $\eprime (H^\prime )\neq 0$ and (3) is fulfilled. 
Assume $H$ fulfills (3). Let $H^\prime$ be the cokernel of the map $\ell \eprime (H) \to H$. Since $e$ is exact, we get $\eprime (H^\prime )=0$, so by (2) it follows $H^\prime =0$. 
Assume $H$ fulfills (2). Let $G$ be an object with $\eprime (G)=0$. Using (2) we get a monomorphism 
\[ 
0\to \Hom_{\mcB} (H,G) \to \Hom_{\mcB}(\ell \eprime (H), G)
\]
and by the adjointness $\Hom_{\mcB} (\ell \eprime (H), G) =\Hom_{\mcC} (\eprime (H), \eprime (G))=0$ and therefore $\Hom_{\mcB} (H,G)=0$, so (1) is fulfilled. 
Now, (4) implies (2) by \cite[Proposition 4.9]{FP}. On the other hand, from loc.\ cit.\ Proposition 4.2 we know that the cokernel of $\ell\eprime (H)\to H$ is precisely $iq (H)$. So, if it is an epimorphism, then $iq (H)=0$ and since $i$ is fully faithful, it follows $q(H)=0$. 
\end{proof}

\begin{dfn}
\label{d:bistable}
We say $F$ is \emph{bistable} if it is stable and costable. 
\end{dfn}

It follows directly from Lemma \ref{IntExt} that $F$ is bistable if and only if $F\cong c\eprime (F)$. 
Now, every recollement of abelian categories with middle term $\mcB$ is determined by its associated TTF-triple, which is a triple $(\mcX , \mcY , \mcZ)$ of subcategories of $\mcB$  such that $(\mcX , \mcY)$ and $(\mcY , \mcZ)$ are torsion pairs, compare e.g. \cite{PV}.
For our given recollement, the TTF-triple is 
\[(\ker q , \Ker \eprime , \Ker p)= (\text{costables}, \Ker \eprime , \text{stables}).\]

\subsection*{Functor categories}
We are interested in \emph{Krull-Schmidt categories}, by which we mean a small additive $K$-category $\mcR$
with finite-dimensional Hom sets and split idempotents.
We denote by $\RKmod$ the category of $K$-linear contravariant functors 
$\mcR\to \text{$K$-$\Modu$}$. 
If $X$ is an object in $\mcR$, 
then the representable functor $\Hom_{\mcR}(-,X)$ is a projective object of $\RKmod$,
and projective objects isomorphic to one of this form are said to be \emph{finitely generated}.
The functor $D\Hom_{\mcR}(X,-)$ is an injective object of $\RKmod$, and
injective objects of this form are said to be \emph{finitely cogenerated}.

Any $K$-linear functor $f\colon \mcS\to\mcR$ of Krull-Schmidt categories induces a restriction functor
$\RKmod\to\SKmod$.
Using tensor products over categories, and the usual hom-tensor adjointness,
one obtains left and right adjoints $\ell,r\colon \SKmod\to\RKmod$ given by
\[
\ell(F)(X) = \Hom_{\mcR}(X,f(-))\otimes_{\mcS} F,
\quad
r(F)(X) = \Hom_{\SKmod}(\Hom_{\mcR}(f(-),X),F)
\]
for $F\in \SKmod$ and $X\in \mcR$.

Recall that if $\mcS$ is a full Krull-Schmidt subcategory of $\mcR$, 
then $\mcR /\mcS$ denotes the quotient category whose objects are the same as in 
$\mcR$ and with morphisms given by the quotient vector space
\[ 
\Hom_{\mcR/\mcS} (X, X^\prime ) := \Hom_{\mcR}(X,X^\prime)/ I_{\mcS}(X, X^\prime)
\]
where $I_{\mcS}(X,X^\prime)$ is the vector subspace with elements 
the morphisms $X\to X^\prime$ which factor through an object in $\mcS$. 
In this situation we have a recollement of the following form.

\begin{lem}
For $\mcR$ a Krull-Schmidt category and $\mcS$ a full additive subcategory of $\mcR$,
there is a recollement
\[ 
\xymatrix{
\RSKmod \ar[rr]|{i} \;  && \; \RKmod \; \ar@<2ex>[ll]^{p} \ar@<-2ex>[ll]_{q} \ar[rr]|{\eprime}  && \; \SKmod \ar@<2ex>[ll]^{r} \ar@<-2ex>[ll]_{\ell}.
}
\]
where $\eprime$ is the restriction functor given by the 
inclusion $\mcS\subset \mcR$
and $i$ is given by composition with the natural functor $\mcR\to \mcR/\mcS$. 
\end{lem}


\section{The categories $\mcQ$ and $\mcH$ and two auxiliary recollements}
\label{s:qhdef}


Let $A$ be a basic finite-dimensional $K$-algebra and
let $e_1,\dots,e_n$ be a complete set of primitive orthogonal idempotents in $A$.
Since $A$ is basic, the modules $P_i = Ae_i$ are a complete set of non-isomorphic indecomposable
projective $A$-modules and their tops $S_i$ are a complete set of non-isomorphic simple $A$-modules.


\subsection*{The category $\mcQ$ of quotients of projectives and a first auxiliary recollement}

Let $\mcQ$ be the category defined as in the introduction.
Let $\mathcal{E}$ be the full subcategory of $\mcQ$ consisting of the objects $1\colon P\to P$,
and let $\mathcal{P}$ be the full subcategory of $\mcQ$ consisting of the objects $0\colon P\to 0$.
The following is clear.

\begin{lem}\label{initial}
The category $\mcQ$ is Krull-Schmidt and its indecomposable objects are of the following three types up to isomorphism
\begin{itemize}
\item[(1)] a projective cover $f_U\colon P_U \to U$ of a non-projective indecomposable $A$-module $U$,
\item[(2)] indecomposables in $\mcE$, so of the form $1\colon P_i\to P_i$ for $1\leq i\leq n$, 
\item[(3)] indecomposables in $\mcP$, so of the form $0\colon P_i \to 0$ for $1\leq i\leq n$. 
\end{itemize}
\end{lem}


\begin{rem}
Clearly $\mcQ$ is a full subcategory of the category $\mcT$ whose 
objects are the morphisms $f\colon Y\to X$ in $A$-$\modu$
and whose morphisms are given by commutative diagrams.
Now $\mcT$ is abelian, indeed it is equivalent to the category of finite-dimensional modules for
the algebra of $2\times 2$ upper triangular matrices with entries in $A$.
It is not difficult to see that $\mcQ$ is
functorially finite and extension-closed  in $\mcT$. Thus it is an exact category and,
by \cite[Theorem 2.4]{AS}, 
the category $\mcQ$ has Auslander-Reiten sequences.
It is easy to see that 
the indecomposable Ext-projectives are the objects of the form $1\colon P_i\to P_i$ and $0\colon P_i\to 0$,
and the indecomposable Ext-injectives are the objects of the form $f_{I_i}\colon P_{I_i}\to I_i$ and $0\colon P_i\to 0$, where $I_i$ is the injective envelope of $S_i$. 
\end{rem}

Considering $\mcP$ as a full subcategory of $\mcQ$,
the first auxiliary recollement we consider is
\[ 
\xymatrix{
\QPKmod \ar[rr]|{i'} \;  && \; \QKmod \; \ar@<2ex>[ll]^{p'} \ar@<-2ex>[ll]_{q'} \ar[rr]|>>>>>>>>>{e'}  && \; 
\text{$A$-$\Modu$} \ (\cong \PKmod)
\ar@<2ex>[ll]^{r'{\!\!\quad\quad \quad }} \ar@<-2ex>[ll]_{\ell'{\!\!\quad\quad \quad }}.
}
\]
The functor $e'$ is given by $e'(F) = F(A\to 0)$ with its induced $A$-module structure.

\begin{thm}
The adjoints $\ell',r'\colon \text{$A$-$\Modu$}\to\QKmod$ of $e'$ are given by
\[ 
\begin{aligned}
\ellq (M) (f\colon P\to X) 
&=
\Hom_A (P,M) \\
\rQ(M) (f\colon P\to X)
&=
\Hom_A(\Ker f,M)
\end{aligned} 
\]
and the intermediate extension functor $c'\colon \text{$A$-$\Modu$}\to\QKmod$ is given by
\[ 
c' (M)(f\colon P\to X) = \coKer (\Hom_A(X,M) \to \Hom_A(P,M)). 
\]
\end{thm}

\begin{proof}
For the moment let $\ell'$, $r'$ and $c'$ be defined by the formulas in the statement of the theorem.
Since $P$ is finitely generated and projective, $\ell'$ is right exact and commutes with direct sums.
Thus, to show it agrees with the left adjoint, it suffices to check this on the free module $A$, that is, 
$\Hom_{\mcQ}(\ell'(A),F)\cong \Hom_A(A,e'(F))$.
This follows from Yoneda's Lemma, since $\ell'(A) \cong \Hom_{\mcQ}(-,(A\to 0))$.

Since $e'$ is right exact and commutes with direct sums, by considering a projective presentation
of $F\in\QKmod$ by direct sums of representable functors, in order to prove that $r'$ is right
adjoint to $e'$ it suffices to prove that
$\Hom_{\QKmod} (R,r'(M)) \cong \Hom_A (e'(R),M)$
for any representable functor $R = \Hom_{\QKmod}(-,(f\colon P\to X))$.
This is clear since on the left hand side Yoneda's Lemma
gives
\[
\Hom_{\QKmod} (R,r'(M)) \cong r'(M)(f\colon P\to X) = \Hom_A(\Ker f,M)
\]
and on the right hand side we have
$e'(R) = \Hom_{\QKmod}((A\to 0),(f\colon P\to X)) \cong \Ker f$.

Now the isomorphism
\[
\Bild(\Hom_A(P,M)\to \Hom_A(\Ker f,M)) \cong \coKer (\Hom_A(X,M) \to \Hom_A(P,M)).
\]
gives the result for $c'$.
\end{proof}

\subsection*{The homotopy category $\mcH$ and a second auxiliary recollement}

We define $\mcH = \mcQ / \mcE$.
It is easy to see that a morphism $(\theta,\phi)$
from the object $f\colon P\to X$ to the object $f'\colon P'\to X'$ 
in $\mcQ$ factors through $\mcE$ if and only if
there is a map $h\colon X\to P'$ with $\theta = h f$
and $\phi = f' h$.
Thus, if one considers $\mcQ$ as a category of complexes with two terms, then $\mcH$ is the
corresponding homotopy category.
The category $\mcH$ is Krull-Schmidt and up to isomorphism its indecomposable objects are those of types (1) and (3)
in Lemma~\ref{initial}.
Associated to the quotient $\mcH = \mcQ/\mcE $ there is a second auxiliary recollement
\[ 
\xymatrix{
\HKmod \ar[rr]|{i_0} \;  && \; \QKmod \; \ar@<2ex>[ll]^{p_0} \ar@<-2ex>[ll]_{q_0} \ar[rr]|>>>>>>>>>{e_0}  && \; 
\text{$A$-$\Modu$} \ (\cong \EKmod)
\ar@<2ex>[ll]^{r_0{\!\!\quad\quad \quad }} \ar@<-2ex>[ll]_{\ell_0{\!\!\quad\quad \quad }}.
}
\]


\begin{lem} \label{qzero-pzero}
The functor $q_0\colon \QKmod\to \HKmod$ sends
a representable functor
$\Hom_{\QKmod}(-,(P\to X))$ to the representable
functor
$\Hom_{\HKmod}(-,(P\to X))$.
\end{lem}

\begin{proof}
For $G\in \HKmod$ we have
\[
\begin{split}
\Hom_{\HKmod}(q_0(\Hom_{\mcQ}(-,(P\to X))),G)
\cong
\Hom_{\QKmod}(\Hom_{\mcQ}(-,(P\to X)),i_0(G))
\\
\cong
i_0(G)(P\to X)
\cong
G(P\to X)
\cong
\Hom_{\HKmod}(\Hom_{\mcH}(-,(P\to X)),G)
\end{split}
\]
giving the claim.
\end{proof}

\begin{lem}\label{monoInH}
If $(\theta,\phi)$ is a morphism in $\mcQ$ and $\theta$ is a split monomorphism,
then $(\theta,\phi)$ induces a monomorphism in $\mcH$.
\end{lem}

\begin{proof}
Say $(\theta,\phi)$ is a morphism from $f\colon P\to X$ to $f'\colon P'\to X'$, 
and let $r\colon P'\to P$ be a retraction for $\theta$, so $r \theta = 1_P$.
Take a morphism in $\mcH$ from an object $f''\colon P''\to X''$ to $f\colon P\to X$, and let it be represented by
a morphism $(\theta',\phi')$ in $\mcQ$. If the composition with $(\theta,\phi)$ is zero in $\mcH$,
then in $\mcQ$ the composition $(\theta \theta',\phi \phi')$ factors through an object in $\mcE$,
say as $(\alpha,\beta)$ from $f''\colon P''\to X''$ to $1\colon Q\to Q$ composed with $(\alpha',\beta')$ from $1\colon Q\to Q$
to $f'\colon P'\to X'$. 
Then $(\theta',r \alpha' \beta)$ defines a map from $f\colon P''\to X''$ to $1\colon P\to P$, and its
composition with the map $(1,f)$ from $1\colon P\to P$ to $f\colon P\to X$ is equal to $(\theta',f r \alpha' \beta)$.
But $(\theta',\phi')$ is also a morphism from $f''\colon P''\to X''$ to $f\colon P\to X$, and since these two maps have
the same first component, and $f''$ is surjective, they must be equal, $(\theta',f r \alpha' \beta) = (\theta',\phi')$.
This shows that $(\theta',\phi')$ is the zero map in $\mcH$, as required.
\end{proof}

\begin{lem}
\label{l:pdlemma}
If $G$ is a functor in $\HKmod$ and $i_0(G)$ is finitely presented in $\QKmod$, 
then $\pdim_{\QKmod} i_0(G) \le 2$.
Moreover $G$ is finitely presented in $\HKmod$ and $\pdim_{\HKmod} G \le 2$. 
\end{lem}

\begin{proof}
By assumption there is a projective presentation
\[
\Hom_{\mcQ}(-,(f'\colon P'\to X')) \to
\Hom_{\mcQ}(-,(f''\colon P''\to X'')) \to
i_0(G) \to 0
\]
for some morphism $(\theta',\phi')$ from the object $f'\colon P'\to X'$ to the object $f''\colon P''\to X''$ in $\mcQ$.
Now the functor $i_0(G)$ vanishes on every object of $\mcE$, so in particular on $1\colon P''\to P''$,
so the map
\[
\Hom_{\mcQ}((1\colon P''\to P''),(f'\colon P'\to X')) \to
\Hom_{\mcQ}((1\colon P''\to P''),(f''\colon P''\to X'')) 
\]
is onto. Thus $(1,f'')$ lifts to a map from $1\colon P''\to P''$ to $f'\colon P'\to X'$.
In particular $\theta'$ is a split epimorphism. Thus $\phi'$ is also an epimorphism, and we 
obtain the diagram 
\[
\begin{CD}
0 @>>> P @>\theta >> P' @>\theta'>> P'' @>>> 0 \\
& & @Vf VV @Vf' VV @Vf'' VV \\
0 @>>> X @>\phi >> X' @>\phi' >> X'' @>>> 0 \\
\end{CD}
\]
where $P = \Ker \theta'$ and $X = \Ker \phi'$.
Now the map $f\colon P\to X$ need not be onto, but factorizing it as $f_0\colon P\to \Bild(f)$
followed by the inclusion, we get an
exact sequence
\[
\begin{split}
0 \to
&\Hom_{\mcQ}(-,(f_0\colon P\to \Bild(f))) \to
\Hom_{\mcQ}(-,(f'\colon P'\to X')) 
\\
\to
&\Hom_{\mcQ}(-,(f''\colon P''\to X'')) \to
i_0(G) \to 0.
\end{split}
\]
Applying $q_0$ one obtains a sequence of functors
\[
\begin{split}
0 \to
&\Hom_{\mcH}(-,(f_0\colon P\to \Bild(f))) \to
\Hom_{\mcH}(-,(f'\colon P'\to X')) 
\\
\to
&\Hom_{\mcH}(-,(f''\colon P''\to X'')) \to
G \to 0
\end{split}
\]
in $\HKmod$. 
Since $q_0$ is a left adjoint, this sequence is
right exact.
To show it is exact at 
the term $\Hom_{\mcH}(-,(f'\colon P'\to X'))$,
we consider a morphism $(\alpha,\beta)$
from $g\colon Q\to M$ to $f'\colon P'\to X'$.
which is sent to zero in
$\Hom_{\mcH}((g\colon Q\to M),(f''\colon P''\to X''))$.
Thus there is a map $h\colon M\to P''$
with $h g = \theta' \alpha$
and $f'' h = \phi' \beta$.
Let $s$ be a section for $\theta'$ and $r$ a retraction for $\theta$ with $1_{P'} = \theta r + s \theta'$.
Then $\alpha-shg = \theta\alpha'$ for some $\alpha'\colon Q\to P$
and $\beta - f'sh = \phi\beta'$ for some $\beta'\colon M\to X$,
which actually has image contained in $\Bild(f)$
since 
\[
\phi \beta' g = \beta g - f'shg = f'\alpha - f's\theta'\alpha = f'\theta r g = \phi f r g
\]
and since $\phi $ is a monomorphism and $g$ is an epimorphism, we have $\beta' = fr$.  
Then, up to a morphism factoring through a projective, $(\alpha,\beta)$ comes from the map
$(\alpha',\beta')$ from $g\colon Q\to M$ to $f_0\colon P\to \Bild(f)$.

Now, because the map from $P$ to $P'$ is a split monomorphism, the map from $f_0\colon P\to \Bild(f)$
to $f'\colon P'\to X'$ induces a monomorphism in $\mcH$, and it follows that the sequence
is exact on the left, too.
\end{proof}

\section{The main recollement}
\label{s:mainrecoll}

Observe that there are no non-zero maps from an object $0\colon P\to 0$ in $\mathcal{P}$ to 
an object $1\colon P'\to P'$ in $\mathcal{E}$,
so the natural functor from $\mcP$ to $\mcH = \mcQ/\mcE$ is fully faithful,
and hence $\mcP$ can be considered as a full subcategory of $\mcH$.
Again, identifying $\PKmod\cong \text{$A$-$\Modu$}$, we have a recollement
\[ 
\xymatrix{
\HPKmod \ar[rr]|{i} \;  && \; \HKmod \; \ar@<2ex>[ll]^{p} \ar@<-2ex>[ll]_{q} \ar[rr]|>>>>>>>>>{e}  && \; 
\text{$A$-$\Modu$} \ (\cong \PKmod) 
\ar@<2ex>[ll]^{r{\!\!\quad\quad \quad } } \ar@<-2ex>[ll]_{\ell{\!\!\quad\quad \quad } }.
}
\]
Clearly $e = e'  i_0$.
We denote by $c\colon \text{$A$-$\Modu$} \to \HKmod$ the corresponding intermediate extension functor.

\subsection*{Projective dimension}
\begin{lem}
For $A$-modules $M$, we have a natural isomorphism
$r'(M) \cong i_0 r(M)$ 
and a natural epimorphism $\ell'(M) \to i_0 \ell(M)$.
\end{lem}

\begin{proof}
For $M$ an $A$-module, the functor $r'(M)$ vanishes on objects in $\mcE$, so it is equal to $i_0 r(M)$
for some functor $r\colon \text{$A$-$\Modu$}\to \HKmod$. Now for $G\in \HKmod$ we have
\[
\begin{split}
\Hom_{\HKmod}(G,r(M)) 
&= \Hom_{\QKmod}(i_0(G),i_0 r(M)) =
\Hom_{\QKmod}(i_0(G),r'(M)) 
\\
&= \Hom_A (e' i_0(G),M) = \Hom_A(e(G),M), 
\end{split}
\]
so this functor is the adjoint.

Now for $M$ in $A$-$\Modu$ and $F\in\HKmod$, using that $e = e' i_0$ we have
\[
\Hom_{\QKmod}(\ell'(M),i_0(F)) 
= \Hom_A(M,e(F))
= \Hom_{\HKmod}(\ell(M),F) = \Hom_{\QKmod}(i_0\ell(M),i_0(F)).
\]
The identity map on the right hand side for $F=\ell(M)$ gives a natural morphism
$\ell'(M)\to i_0 \ell(M)$. 

Now for any $G\in\QKmod$, 
by \cite[Proposition 4.2]{FP} the natural map $i_0 p_0(G)\to G$
is a monomorphism. Since $e'$ is exact and $e=e' i_0$, we deduce that for any $A$-module $M$, the map 
\[
\Hom_A(M,e p_0(G))\to \Hom_A(M,e'(G))
\]
is a monomorphism.
By adjointness we can rewrite this as 
\[
\Hom_{\QKmod}(i_0 \ell(M),G) \hookrightarrow \Hom_{\QKmod}(\ell'(M),G)
\]
giving the result.
\end{proof}

\begin{lem}
We have $c' = i_0 c$, so that $c$ is given by the same formula as $c'$, that is,
\[
c(M)(f\colon P\to X) = \coKer \left(\Hom_A(X,M) \to \Hom_A(P,M) \right).
\]
\end{lem}

\begin{proof}
For an $A$-module $M$, recall that $c'(M)$ is the image of the map from $\ell'(M)$ to $r'(M)$. 
Now, using that $i_0$ is exact, this factors as
\[
\ell'(M) \twoheadrightarrow i_0 \ell(M) \twoheadrightarrow i_0 c(M)
\hookrightarrow i_0 r(M) = r'(M).
\]
\end{proof}

\begin{lem}
\label{l:cpd}
Let $M$ be a finite-dimensional left $A$-module and let $p\colon Q\to M$ be a projective cover.
Then $\cp (M)$ has a projective resolution
\[
0\to \Hom_{\mcH}(- , (p\colon Q\to M))\to \Hom_{\mcH}(- , (0\colon Q\to 0) ) \to \cp (M)\to 0
\]
in $\HKmod$. In particular $\cp(M)$ is finitely presented and $\pdim_{\HKmod} (\cp (M)) \leq 1$.
\end{lem}

\begin{proof}
Let $f\colon P\to X$ be an object in $\mcQ$.
By diagram chasing one obtains an exact
sequence
\[
0\to
\Ker\left(\Hom_A(P,Q)\oplus \Hom_A(X,M)\to \Hom_A(P,M)\right) \to \Hom_A(P,Q) \to c'(M)(P\to X) \to 0.
\]
This gives an exact sequence
\[
0\to
\Hom_{\mcQ}(-,(Q\to M))
\to
\Hom_{\mcQ}(-,(Q\to 0))
\to
c'(M)\to 0.
\]
Applying $q_0$ and using that $q_0$ preserves representables by Lemma~\ref{qzero-pzero}
\[
0\to
\Hom_{\mcH}(-,(Q\to M))
\to
\Hom_{\mcH}(-,(Q\to 0))
\to
c(M)\to 0.
\]
It is right exact since $q_0$ is a left adjoint, and it is exact on the left by Lemma~\ref{monoInH} about monomorphisms in $\mcH$.
\end{proof}

\subsection*{Rigidity}
In general, we do not know when intermediate extension functors map all modules to rigid modules but in our situation, we can prove this statement, generalizing \cite[Theorem 5.6]{CIFR}.

\begin{thm} 
\label{rigid}
We have $\EXT^1_{\HKmod}(\cp (M), \cp (N) )=0$ for finite-dimensional left $A$-modules $M$, $N$.
\end{thm}

\begin{proof} 
Applying $\Hom_{\HKmod} (-, \cp (N))$ to the projective resolution for $\cp(M)$
gives an exact sequence 
\[
\begin{aligned}
0 & \to \Hom_{\HKmod} (\cp (M), \cp (N))\to \Hom_{\mcH}( (-, Q\to 0), \cp (N)) \\
& \to \Hom_{\mcH} ( (-, Q\to M), \cp (N)) \to \EXT^1_{\HKmod}(\cp (M), \cp (N)) \to 0.
\end{aligned}
\]
Using the fully faithfulness of $\cp $ we can write it as 
\[ 
0\to \Hom_A(M,N) \to \cp(N)(Q\to 0)\to \cp(N)(Q\to M)\to \EXT^1_{\HKmod}(\cp(M), \cp(N))\to 0
\]
and by the definition of $\cp(N)$ the middle map is
\[ 
\Hom_A (Q,N) \to 
\coKer (\Hom_A (M,N) \to \Hom_A (Q,N)) 
\]
which is tautologically surjective, giving the result.
\end{proof}

We observe the following special property that $\cp$ maps injectives to injectives and projectives to projectives. More precisely, one has

\begin{lem} \label{cproj-cinj}
Let $M$ be an $A$-module.
\begin{itemize}
\item[(1)] 
If $M$ is injective, then
$\cp (M)=r(M)$ is injective in $\HKmod$.
\item[(2)] 
If $M$ is projective, then
$\cp (M)=\ell (M) $ is projective in $\HKmod$. 
\end{itemize}
\end{lem}

\begin{proof}
(1) If $M$ is injective, then
\[
c(M)(f\colon P\to X) = \coKer \left(\Hom_A(X,M) \to \Hom_A(P,M) \right) 
\cong \Hom_A(\Ker f,M) \cong r(M)(f\colon P\to X),
\]
so $c(M)\cong r(M)$. Now the functor
$\Hom_{\HKmod}(-,r(M)) \cong \Hom_A(e(-),M)$
is exact since $e$ is exact and $M$ is injective, so $r(M)$ is injective.

(2) It suffices to prove this for $M$ a finitely generated projective module (or even for $M=A$).
In this case it follows from the projective resolution of $c(M)$ since one can take $Q=M$ and the first term is zero.
\end{proof}

We remark that $c'$ sends injectives to injectives, but it need not send projectives to projectives.

\subsection*{Injective dimension}
The Nakayama functors $\nu$ and $\nu^-$ 
for $A$-modules are defined by $\nu(M) = D\Hom_A(M,A)$
and $\nu^-(M) = \Hom_A(DM,A)$.
They define inverse equivalences between the category of finite-dimensional projective $A$-modules
and the category of finite-dimensional injective $A$-modules.
For any modules $X$ and $Y$ there is a functorially defined map
\[
\Theta_{X,Y}\colon D\Hom_A(X,Y) \to \Hom_A(Y,\nu(X))
\]
which is an isomorphism if $X$ is a finite-dimensional projective module.

\begin{lem}
The following triangle is commutative
\[
\xymatrix{
D\Hom_A(\nu(P),\nu(P'))
\ar[rr]^{D\nu} 
\ar[rd]_{D\Theta_{P',\nu(P)}} 
& & 
D\Hom_A(P,P')
\ar[ld]^{\Theta_{P,P'}} 
\\ & 
\Hom_A(P',\nu(P))
&}
\]
for any finite-dimensional projective $A$-modules $P$ and $P'$. 
\end{lem}

\begin{proof}
One can reduce to the case when $P=P'=A$.
\end{proof}

\begin{lem}
\label{l:cid}
If $M$ is a finite-dimensional left $A$-module, then $\cp (M)$ is finitely copresented in $\HKmod$ and
$\idim_{\HKmod} (\cp (M)) \leq 1$. 
\end{lem}

\begin{proof}
Choose a minimal injective copresentation 
$0\to M\to I \to J$.
Applying $\nu^-$ and using the definition of the Auslander-Reiten translate $\tau^-$ 
we get an exact sequence
\[
0\to \nu^- M \to \nu^- I \to \nu^- J \to \tau^-(M)\to 0.
\]
This gives an object $\nu^- J\to \tau^- M$ in $\mcQ$
and a morphism from $\nu^- I\to 0$ to $\nu^- J\to \tau^- M$.

For an object $f\colon P\to X$ in $\mcQ$, consider the sequence
\[
0
\to
\Hom_{\mcH}((\nu^- J\to\tau^- M),(P\to X))
\to
\Hom_{\mcH}((\nu^- I\to 0),(P\to X))
\xrightarrow{h}
\Hom(M,\nu P)
\]
where $h$ sends a map in the left hand space, which can be identified with $\Hom(\nu^- I,\Ker f)$
to the map obtained by first composing with the inclusion $\Ker f\to P$ to get a map $\nu^-I\to P$,
then applying $\nu$ to get a map $I \to \nu P$, and then composing with the inclusion $M\to I$
to get a map $M\to \nu P$, as required.
The sequence is exact in the middle because a map $\nu^- I\to \Ker f$ is sent to 0 by $h$
if and only if the map $I\to \nu P$ factors through the map $I\to J$ (using that $J$ is injective),
so if and only if the map $\nu^- I\to P$ factors through $\nu^- I \to \nu^- J$,
which is the condition for the map 
from $\nu^- I\to 0$ to $P\to X$ to come from a map from
$\nu^- J\to\tau^- M$ to $P\to X$.
The sequence is easily seen to be exact on the left.
Now the dual of $h$ is the composition along the top of the diagram
\[
\begin{CD}
D\Hom(M,\nu P) @>>> D\Hom(I,\nu P) @>>> D\Hom(\nu^- I,P) @>>> D\Hom(\nu^- I,\Ker f) 
\\
D\Theta_{P,M} @| 
D\Theta_{P,I} @| 
\Theta_{\nu^- I,P} @| 
\Theta_{\nu^- I,\Ker f} @|
\\
\Hom(P,M) @>>> \Hom(P,I) @= \Hom(P,I) @>>> \Hom(\Ker f,I) 
\end{CD}
\]
The diagram commutes - the middle square commutes by dualizing a triangle in a lemma above.
The composition along the bottom can also be factorized as
\[
\Hom(P,M)\to \Hom(\Ker f,M) \hookrightarrow \Hom(\Ker f,I)
\]
so the image of the dual of $h$ is isomorphic to the image of the map
$\Hom(P,M)\to \Hom(\Ker f,M)$, which is $c(M)(P\to X)$.
Thus, dualizing the original sequence, we get an exact sequence
\[
0 \to c(M)
\to
D\Hom_{\mcH}((\nu^- I\to 0),-)
\to
D\Hom_{\mcH}((\nu^- J\to\tau^- M),-)
\to
0,
\]
giving the result.
\end{proof}

\subsection*{Stable and costable functors}
We want to characterize stable and costable objects for our main recollement.

\begin{lem} 
\label{l:Fstable}
Let $F$ in $\HKmod$. The following are equivalent. 
\begin{itemize}
\item[(1)] $F$ is stable.
\item[(2)] There is an injective $A$-module $I$ such that $F$ is a subfunctor of $\cp (I)$.
\item[(3)]  There is an\ $A$-module $M$ such that $F$ is a subfunctor of $\cp (M)$.
\end{itemize}
\end{lem}

\begin{proof} 
(3) implies (1) because submodules of stable modules are stable. 
Assume $F$ fulfills (1). Let $I$ be an injective hull of $\eprime (F)$. We apply the left exact functor $\rp$ and get a monomorphism $\rp \eprime (F) \to \rp (I)$. By Lemma~\ref{cproj-cinj} we have $\rp (I) =\cp (I)$ and by Lemma~\ref{stable}, we have a monomorphism $F\to \rp \eprime (F)$. Composition gives a monomorphism $F\to \cp (I)$ and (2) is fulfilled.
Now, (2) implies clearly (3). 
\end{proof}

\begin{lem}
\label{l:Hcostable}
Let $H$ in $\HKmod$. The following are equivalent. 
\begin{itemize}
\item[(1)] $H$ is costable.
 \item[(2)] There is a projective $A$-module $P$ such that $H$ is a quotient of $\cp (P)$.
 \item[(3)]  There is a $A$-module $M$ such that $H$ is a quotient of $\cp (M)$.
\end{itemize}
 \end{lem}

We remark the following. 
\begin{lem}\label{Ext2} 
If
$
0\to F\to \cp (M) \to H \to 0
$
is a short exact sequence in $\HKmod$, with $M$ an $A$-module,
then $F$ is stable, $\cp (M)$ is bistable and $H$ is costable. 
If $\idim F\leq 2$ or $\pdim H \leq 2$, then  
\[ 
\EXT^2_{\HKmod}(F,F)=0=\EXT^2_{\HKmod}(H,H). 
\]
\end{lem}

\begin{proof}
The first part is straightforward.
Apply $\Hom_{\HKmod}(-, F)$ to the sequence to get a long exact sequence
\[
\cdots \to \EXT^2_{\HKmod}(\cp (M), F) \to \EXT^2_{\HKmod}(F,F)\to \EXT^3_{\HKmod}(H,F)\to \cdots. \]
Now $\EXT^2_{\HKmod}(\cp(M), F)=0$ by Theorem \ref{rigid}, and
$\EXT^3_{\HKmod}(H,F) = 0$ by the hypothesis on injective or projective dimension,
so $\EXT^2(F,F)=0$. 
Similarly, to show that
$\EXT^2_{\HKmod}(H,H)=0$, apply $\Hom_{\HKmod}(H,-)$ and use the analogue argument.
\end{proof}

\section{The projective quotient algebra}
\label{s:pqa}
Let $A$ be a basic finite-dimensional algebra as in \S\ref{s:qhdef} and \S\ref{s:mainrecoll}.
We now suppose that $A$ has finite representation type. In this case $\mcH$ has only finitely many
indecomposable objects up to isomorphism, those of the form $P_U\to U$, where $U$ is a
non-projective indecomposable $A$-module and $P_U$ is a projective cover, and those of the form $P_i\to 0$. 
We denote by $G$ the direct sum of all these indecomposable objects
and define $B = B_A = \End_{\mcH}(G)^{op}$. It is a finite-dimensional basic algebra,
and $\mcH$ is equivalent to the category of finitely generated projective $B$-modules.
Now the object $A\to 0$ in $\mcH$ is the direct sum of the objects 
$P_i \to 0$, so it is is a summand of $G$,
and $e\in B$ is the projection of $G$ onto this summand.
Then $eB e \cong \End_{\mcH}(A\to 0)^{op} \cong \End_A(A)^{op} \cong A$.
The recollement given at the start of this section can be identified with
the one stated in the introduction to the paper, given by the idempotent $e\in B$.


Recall that $E$ is the direct sum of all indecomposable $A$-modules and the
Auslander algebra is $\Gamma = \End_A(E)^{op}$. 
We define $C=c(E)$. Since $c$ is fully faithful, we have $\End_B(C)^{op} \cong \Gamma$.

\begin{lem} 
\label{gldim2}
The projective quotient algebra $B$ has global dimension at most $2$
and the module $C$ is a tilting and cotilting module.
\end{lem}

\begin{proof}
Any finite-dimensional $B$-module has projective dimension $\le 2$ by
Lemma~\ref{l:pdlemma}, giving the global dimension claim.
The number of indecomposable summands of $C$ is equal to the number of
indecomposable $A$-modules, which is also the number of indecompsable objects in $\mcH$,
so the number of simple $B$-modules. Thus
$C$ is a tilting and cotilting module by Lemmas~\ref{l:cpd}, \ref{l:cid} and Theorem~\ref{rigid}.
\end{proof}

\begin{thm}
\label{t:cformula}
The intermediate extension functor $c:\text{$A$-$\Modu$} \to \text{$B$-$\Modu$}$
is given by $c(X) = C \otimes_\Gamma \Hom_A(E,X)$.
It satisfies $\Ext^1_B(c(X),c(Y)) = 0$ for any $A$-modules $X,Y$.
\end{thm}

\begin{proof}
The assignment $\Hom_A(E,X)\to \Hom(c(E),c(X))$ induces a map
$h_X:c(E)\otimes_\Gamma \Hom_A(E,X)\to c(X)$ which is functorial in $X$. Now $h_E$ is an
isomorphism, both sides commute with direct sums, and any $A$-module is
a direct sum of summands of $E$, so $h_X$ is an isomorphism for all $X$.
The assertion about extensions follows from Theorem~\ref{rigid} since every
$A$-module is a direct sum of finite-dimensional $A$-modules.
\end{proof}

The next lemma follows from Lemmas~\ref{l:Fstable} and \ref{l:Hcostable}.

\begin{lem}  
The stable $B$-modules are those cogenerated by $C$ and the
costable modules are those generated by $C$. 
\end{lem}

We now recall some results from tilting theory (which apply to the projective quotient algebra). 
For the moment let $B$ be an arbitrary finite dimensional algebra.  A tilting and cotilting $B$-module $C$ induces two torsion pairs on $B$-$\modu$ 
\[ 
\left(\mcT=\Gen (C), \mcF= \Ker \Hom_B(C,-)\right), \left(\mcX= \Ker \Hom_B(-,C) , \mcY=\Cogen (C)\right),   
\]
the first from the tilting property, the second from the cotilting property. Because $C$ is cotilting all  injective modules are objects in $\Gen (C)$ and because $C$ is tilting all projective objects are in $\Cogen (C)$. We set $\Gamma = \End_B (C)^{op}$, then $\Dual C$ is a tilting and cotilting (left) $\Gamma $-module
and by the Brenner-Butler theorem the two torsion pairs in $B$-$\modu$ in the order above are mapped under (four different) equivalences to the subcategories of $\Gamma$-$\modu $
\[ 
 \left( \mcY^\prime =\Cogen (\Dual C), \mcX^\prime =\Ker \Hom_{\Gamma} (-, \Dual C)\right)  ,\left( \mcF^\prime = \Ker \Hom_{\Gamma} ( \Dual C,-), \mcT^\prime =\Gen (\Dual C)\right) 
\]   
respectively (swapping torsion pairs and torsion and torsion-free classes).  
The first two equivalences hold since $C$ is tilting, the second two since $C$ is cotilting. 

This leads to another characterization of 
projective quotient algebras.
Let $B$ again be the projective quotient algebra.

\begin{thm}
\label{t:kerhomaddcond}
We have
$\Ker\Hom_B(C,-) = \Ker\Hom_B(-,C)$
and 
$\Gen (C)\cap \Cogen (C) =\Add (C)$.
\end{thm}

Here $\Add (C)$ is the full subcategory of $B$-$\Modu$ of summands of direct sums of copies of~$C$.

\begin{proof}
The TTF-triple of our recollement (see after Definition~\ref{d:bistable}) is given by 
\[ 
(\Gen (C), \Ker \Hom_B(C,-) =\Ker \eprime =\Ker \Hom_B(-,C) , \Cogen (C) ),
\]
giving the first equality, $\mcF = \mcX$.
The second follows from 
the discussion after Definition~\ref{d:bistable}.
\end{proof}

Recall from the introduction that there is a unique basic tilting  cotilting $\Gamma$-module $T$ generated and cogenerated by projective-injectives.

\begin{thm}
\label{t:dct}
We have $\Dual C \cong T$, so $B \cong \End_\Gamma(T)$.
\end{thm}

\begin{proof}
By Lemma~\ref{l:introlem} it suffices to show that $\Dual C$ is 
generated and cogenerated by projective-injective left $\Gamma$-modules,
or equivalently that $C$ is generated and cogenerated by projective-injective
right $\Gamma$-modules.
By the theory of Auslander algebras (or the Nakayama isomorphism $\Theta_{X,Y}$),
$E = e(C)$ is a projective-injective right $\Gamma$-module. 
We look at the epimorphism $p$ and monomorphism $i$ 
\[ 
\ell e(C) \xrightarrow{p} c e(C) =C \xrightarrow{i} re (C). 
\]
These are maps of left $B$-modules, but they are functorial, so also
maps of right $\Gamma$-modules. By definition of the right adjoint, $re (C) = \Hom_A (e(B), e(C))$ is a $\Gamma$-submodule of $\Hom_K(e(B), e(C))\in \Add (e(C)) =\Add (E)$. By definition of the left adjoint, $\ell e (C) = Be \otimes_A e(C)$, which is a quotient of $Be \otimes_K e(C)\in \Add (E) $ as a $\Gamma $-module. This proves that $C$ is generated and cogenerated by $\Add (E)$.
\end{proof}

The same argument gives another characterization of
projective quotient algebras.

\begin{thm}
If $B$ is an arbitrary finite-dimensional algebra
and $C$ is a basic tilting and cotilting $B$-module 
satisfying the properties in Theorem~\ref{t:kerhomaddcond},
then $B$ is a projective quotient algebra.
\end{thm}

\begin{proof}
The condition on the kernels gives a strong TTF triple 
\[
(\mcX = \Gen (C) , \mcY = \Ker \Hom_B (C, - ) = \Ker \Hom_B (-, C), \mcZ = \Cogen (C)) 
\]
in $B$-$\Modu$. 
The corresponding recollement is given by an idempotent element $e\in B$ by \cite[Corollary 5.5]{PV}. 
Letting $A = eBe$, the intermediate extension functor gives an equivalence 
\[ 
c\colon \text{$A$-$\Modu$} \to \Gen (C) \cap \Cogen (C) = \Add (C)
\]
from which it follows that $A$ is representation-finite and $E=e(C)$ is an Auslander generator. This implies that
\[ 
\Gamma = \End_B (C)^{op} = \End_A (E)^{op} =\End_A ( e(C))^{op}
\]
is the Auslander algebra of $A$.
Now $C$ is generated and cogenerated by projective-injective right $\Gamma$-modules as in the proof of Theorem~\ref{t:dct}. The result follows.  
\end{proof}

Now, let $\Gamma$ be the Auslander algebra and $\epsilon \in \Gamma = \End_A(E)^{op}$ 
be the idempotent element corresponding to the projection on the direct sum of all projective $A$-modules 
(which is isomorphic to $A$, so $\epsilon \Gamma \epsilon \cong \End_A(A)^{op} \cong A$). 
We look at the map given by left multiplication with this 
idempotent,
\[ 
\epsilon\colon \text{$\Gamma$-$\Modu$} \to \text{$A$-$\Modu$}. 
\]
We remark that, $\Ker \epsilon $ is equivalent to 
$\left( \underline{\text{$A$-$\modu$}}\right)^{\wedge } $, which is also equivalent to 
$\Ker \eprime \cong \HPKmod$ in our main recollement. It means we can find two recollements with the same outer terms and modules over the two tilted algebras $B$ and $\Gamma$ in the middle. 
The next result shows that $T$ appears naturally in the recollement
given by the idempotent $\epsilon \in \Gamma$.

\begin{pro} 
One has $\Ker \epsilon =\Ker \Hom_{\Gamma} (- ,T) $.
In particular, we can find $\Add (T)$ as the Ext-injectives in the corresponding torsionfree class to $\Ker \epsilon$.  
\end{pro}

\begin{proof}
Using the previous remark and the Brenner Butler theorem, we know that $\Ker \epsilon \cong \Ker \eprime =\Ker \Hom_B(C,-)\cong \Ker \Hom_{\Gamma} (- ,T)$. Therefore, 
it is enough to prove one inclusion. So, let $X$ be a $\Gamma$-module with 
$\Hom_{\Gamma }(X, T)=0$. Since $T$ is a tilting module, it has as a direct summand every finite dimensional projective-injective module. So in particular, every $I(a)$ with $a$ being a simple $\Gamma$-module corresponding to a projective $A$-module is a summand of $\Gamma$. Therefore, $\Hom_{\Gamma}(X, I(a))=0$ for those $a$. But this implies 
$\dim_a X=0$ for those $a$ and therefore $\epsilon X=0$. 
\end{proof}

\section{Representation schemes and varieties}
In this section we prove some general results about representation varieties and schemes for algebras.
We prove a version of the First Fundamental Theorem for quivers, and then we apply it to finite-dimensional algebras. Then we prove a criterion for smoothness in representation schemes. These results are used in later sections.

We begin by recalling the basic definitions.
Let $K$ be an algebraically closed field, let $A$ be a finitely generated $K$-algebra
and let $e_1,\dots,e_n$ be a complete set of orthogonal idempotents in $A$.
For $\dd\in \N_0^n$, there is an affine scheme $\RRdA$ of $\dd$-dimensional representations 
of $A$.
For a commutative $K$-algebra $R$, its $R$-points are the $K$-algebra homomorphisms 
$\theta:A\to M_d(R)$
where $d = \sum_{i=1}^n d_i$, with the property that $\theta(e_i) = \epsilon_i$ for all $i$, 
where $\epsilon_i$ is the idempotent in
$M_d(K)$ which, under the isomorphism
\[
M_d(K) \cong \End_K(\bigoplus_{i=1}^n K^{d_i}),
\]
corresponds to projection onto the summand $K^{d_i}$.

The $K$-points of $\RRdA$ are $K$-algebra maps $\theta:A\to M_d(K)$, and 
the corresponding $A$-module is $X = K^d$ with the action of $a\in A$ given by left
multiplication by $\theta(a)$. The condition that $\theta(e_i)=\epsilon_i$ ensures
that $e_i X = K^{d_i}$ under the identification $K^d = \bigoplus_{i=1}^n K^{d_i}$.
We write $\RdA$ for the affine variety of $\dd$-dimensional representations of $A$. Its points
are the $K$-points of $\RRdA$, ignoring any non-reduced structure on the scheme.
The algebraic group $\Gd$ acts on by conjugation, and its orbits are the isomorphism
classes of representations of $A$ of dimension $\dd$, that is, $A$-modules $X$ with $\dim e_i X = \alpha_i$ for all $i$.

Suppose that $K$ has characteristic zero. For an algebraic group $G$ 
acting on an affine variety $X$, we denote by $X\dbslash G$ the
categorical quotient, the variety with coordinate ring $K[X]^G$, where $K[X]$ is the
coordinate ring of $X$.

\subsection*{First Fundamental Theorem for quivers with partial group action}
\label{deframedquiver}
Let $Q_B$ be a finite quiver with vertices $1,\ldots , n, n+1, \ldots , m$ ($n \le m$)
and let $(\dd, \rr) \in \N_0^n  \times \N_0^{m-n} $ be a dimension vector for $Q$. 
Let $Q^{>n}$ be the full subquiver of $Q_B$ supported on the vertices $n+1, \ldots , m$. 
We call a non-trivial path in $Q_B$ primitive if starts and ends at a vertex in the range $1, \ldots , n$ but does not otherwise pass through any such vertex. 
Let $e:= \sum_{i=1}^n e_i$, and observe that $e KQ_B e$ is isomorphic to the path algebra of a quiver $Q_A'$ with vertices $1, \ldots , n$ whose arrows are the primitive paths. Note that $Q_A'$
may be infinite, but if we only consider primitive paths of length at most $N^2$, where 
$N= 1+\sum_{i=n+1}^m r_i$, we obtain a finite subquiver $Q_A$ of $Q_A'$.

\begin{lem}
\label{l:quivedeframe}
The natural map
$
\RdrQB \dbslash \Glr \to \RdQA \times \left(\RrQn \dbslash \Glr \right)
$
 is a closed immersion. 
\end{lem}

We use the method of deframing, compare 
\cite[\S 1]{CBmm} and the proof of \cite[Proposition 3.1]{CFR}. 

\begin{dfn} \label{deframed} Let $Q^{\infty}$ be the quiver with vertices $\infty , n+1, n+2, \ldots , m$ and arrows as follows. For  each arrow $a\colon i\to j$ in $Q_B$ we have 
\begin{itemize}
\item[(1)] If $i,j \in \{ n+1, \ldots ,m \}$, then associated to $a$ there is one arrow $i\to j$ in $Q^{>n}$. 

\item[(2)] If $i\in \{1, \ldots ,n\} $, $j \in \{ n+1, \ldots ,m \}$ then associated to $a$ there are $d_i$ arrows $\infty \to j$ in $Q^\infty $.

\item[(3)] If $i \in \{ n+1, \ldots ,m \}$, $j\in \{1, \ldots , n\} $ then 
associated to $a$ there are $d_j$ arrows $i \to \infty $ in $Q^\infty $.

\item[(4)] If $i, j \in \{1, \ldots , n\}$, then associated to $a$ there are $d_id_j$ loops at $\infty$.   
\end{itemize}
\end{dfn}

\begin{proof}[Proof of  Lemma \ref{l:quivedeframe}]
By breaking matrices into their rows and columns it is easy to see that there is a $\Glr$-equivariant isomorphism of varieties
\[ 
\RdrQB \to \RrQinfty.
\]
Moreover, the diagonal copy of the multiplicative group $K^*$ in 
$\Gl_{(1,\rr)} \cong K^* \times \Gl_{\rr}$ acts trivially on $\RrQinfty$,
so 
\[
\RrQinfty \dbslash \Gl_{(1, \rr)}\cong \RrQinfty \dbslash \Glr.
\]
We have to prove that the map on the coordinate rings 
\[ 
\varphi\colon K[ \RdQA ]\otimes_K K[ \RrQn ]^{\Gl_{(1, \rr )}} \to K[\RrQinfty ]^{\Glr}
\]
is surjective. By the First Fundamental Theorem \cite{LP} the ring $K[\RrQinfty ]^{\Gl_{(1, \rr )}}$ is generated by traces along oriented cycles in $Q^\infty$ of length at most $N^2$. Moreover, since the dimension vector at the vertex $\infty $ is equal to $1$, we can reduce to oriented cycles of 
the following two types 

(i) An oriented cycle in passing only through the vertices $n+ 1, \ldots , m$. This invariant is the image under $\varphi $ of the corresponding trace invariant for $Q^{>n}$.  

(ii) An oriented cycle of length at most $N^2$ starting and ending at vertex $\infty $ and not otherwise passing through $\infty $. Each primitive path of length at most $N^2$ from $i$ to $j$ in 
$Q_B$ gives rise to $d_i d_j$ oriented cycles of this type and all arise this way. 
Such a primitive path correspond to an arrow from $i$ to $j$ in $Q_A$ and therefore to a matrix of size $d_j \times d_i$ in $\RdQA$. The maps $\RdQA\to K$ 
picking out the entries of this matrix are sent by $\varphi $ to the required trace invariant. 

Thus each generator of $K[\RrQinfty ]^{\Glr}$ is in the image of $\varphi$, so $\varphi$ is onto.
\end{proof}

\subsection*{First Fundamental Theorem for finite dimensional algebras with an idempotent}
Let $K$ be an algebraically closed field of characteristic zero.
Let $B$ be a basic finite-dimensional algebra and let 
$e_1, \ldots , e_n, e_{n+1}, \ldots , e_m$ ($n\le m$) be a 
complete set of primitive orthogonal idempotents in $B$.
Let $e=\sum_{i=1}^n e_i$ and let $A = eBe$.
For any dimension vector $(\dd, \rr) \in \N_0^m$ we have a natural restriction map
$e\colon \RdrB \to  \RdA$ given by left multiplication with $e$, and it induces a map
$\RdrB \dbslash \Glr \to \RdA$.

\begin{lem}\label{FFT} 
The map $\RdrB \dbslash \Glr \to \RdA$ is a closed immersion.
\end{lem}

Note that here we must use the reduced scheme structure on $\RdA$ and $\RdrB$.

\begin{proof}

We write $B=KQ_B/I$ with $Q_B$ a quiver with vertices $1, \ldots , m$ corresponding to the idempotents $e_1, \ldots , e_m$ in $B$ and $I$ an admissible ideal. 
Then there is a closed immersion $\RdrB \to \RdrQB$.
Now we find a commutative diagram
\[
\xymatrix{
\RdrQB \dbslash \Glr \ar[r] & \RdQA \times \left( \RrQn \dbslash \Glr \right) \\
\RdrB \dbslash \Glr \ar[r]\ar[u]& \RdA \times \left(\textrm{R}_{fBf}(\rr)\dbslash \Glr\right)  \ar[u].
}
\]
where $f= \sum_{i=n+1} ^m e_i$ and the right hand map is induced by the algebra homomorphisms 
$KQ_A \to A$ and $KQ^{>n}\to fBf$.
By Lemma~\ref{l:quivedeframe} the top map is a closed immersion, and the Reynolds
operator ensures that the left hand vertical map is a closed immersion. It follows that the 
bottom map is a closed immersion.
Now, the variety  $\textrm{R}_{fBf}(\rr)\dbslash \Glr $ classifies semi-simple 
$fBf$-modules of dimension vector $\rr$. Since $B$ is basic, so is $fBf$ and therefore this quotient variety is just a point.
The claim follows. 
\end{proof}


\subsection*{Smooth points of module schemes}
Let $A$ be a finitely generated algebra over an algebraically closed field $K$,
let $e_1,\dots,e_n$ be a complete set of orthogonal idempotents in $A$ and let
$\dd\in\N^n$ be a dimension vector.
(Note that we allow the possibility that $n=1$, so we just work with the total dimension of a representation).
Recall that the closed points of the scheme $\RRdA$ are identified
with its $K$-points, so correspond to $A$-module structures $X$ on $\bigoplus_{i=1}^n K^{d_i}$.

\begin{lem} 
\label{Ext2smooth}
If $\Ext^2_A(X,X)=0$, then $X$ is a smooth point of $\RRdA$.
\end{lem}

We use the following version of the lifting criterion for smoothness, see \cite[Tag 02HW]{stacks-project}. 

%

\begin{lem}
Suppose $\Lambda$ is a finitely generated commutative $K$-algebra and $\m$ a maximal ideal in $\Lambda$.
The following are equivalent.
\begin{itemize}
\item[(1)] The scheme $\Spec \Lambda$ is smooth at $\m$.  
\item[(2)] Any $K$-algebra map $\Lambda\to R/I$ with 
\[
(*) \left\{
\begin{aligned}
&\text{$R$ is a finite dimensional commutative local $K$-algebra with}
\\[-4pt]
&\text{maximal ideal $\n$ and $I\subset R$ is an ideal with $I^2=0$ and }
\\[-2pt]
&\text{$I\cong R/\n$ as an $R$-module, and the preimage of $\n/I$ is $\m$,}
\end{aligned}
\right.
\] 
lifts to a $K$-algebra map $\Lambda\to R$. 
\end{itemize}
\end{lem}

\begin{proof}[Proof of Lemma \ref{Ext2smooth}]
The module structure $X$ corresponds to a homomorphism $\theta\colon A \to M_d(K)$
with $\theta(e_i) = \epsilon_i$ for all $i$,
and we can identify $\End_K(X)$ with $M_d(K)$, with the left and right actions of 
$a\in A$ corresponding to multiplication on the left or right by $\theta(a)$.

One knows that $\RRdA = \Spec \Lambda$, for a suitable finitely generated commutative $K$-algebra $\Lambda$ with the property that 
$K$-algebra maps $\Lambda\to R$ correspond to $K$-algebra maps $\theta:A \to M_d(R)$
with $\theta(e_i)=\epsilon_i$ for all $i$.

In the setup of the lifting criterion we are given a finite dimensional commutative local $K$-algebra $R$
with maximal ideal $\n$ and an ideal $I$ with $K$-dimension 1 and $I^2=0$.
Let $p:M_d(R)\to M_d(R/I)$ and $\pi:M_d(R/I) \to M_d(R/\n) \cong M_d(K)$
be the natural maps.
According to the correspondence above, for the criterion we are
given an algebra homomorphism $\phi:A \to M_d(R/I)$ with $\pi \phi = \theta$ and
$\phi(e_i)=\epsilon_i$ for all $i$, and need to lift it to an algebra homomorphism $\psi:A\to M_d(R)$
with $\phi = p\psi$ and $\psi(e_i)=\epsilon_i$ for all $i$.

Now $M_d(I)$ is an ideal in $M_d(R)$, so it becomes a bimodule with left and right actions of $M_d(R)$.
Since $I^2=0$, these actions descend to actions of $M_d(R/I)$, so via $\phi$ they give actions of $A$ on the left and right.
In fact the actions descend to $M_d(R/\n) \cong M_d(K)$, so the actions of $A$ are given by left and
right multiplication by $\theta(a)$.
Thus $M_d(I)$ isomorphic to $\End_K(X)$ as a $A$-$A$-bimodule.
Then $H^2(A,M_d(I)) \cong H^2(A,\End_K(X)) \cong \Ext^2(X,X) = 0$.
Thus any (square zero) algebra extension of $A$ by $M_d(I)$ is trivial
(see for example \cite[\S 9.3]{Weib}).
Consider the pullback diagram
\[
\begin{CD}
0 @>>> M_d(I) @>>> P @>>> A @>>> 0 \\
& & @| @VVV @V\phi VV \\
0 @>>> M_d(I) @>>> M_d(R) @>p >> M_d(R/I) @>>> 0
\end{CD}
\]
where $P = \{ (C,a) \in M_d(R)\oplus A \colon  p(C) = \phi(a) \}$.
Since $p$ and $\phi$ are algebra homomorphisms, $P$ becomes a $K$-algebra, and it is an
extension of $A$ by $M_d(I)$, so it must be the  trivial extension.
Thus the map $P \to A$ has a section $A \to P$. 
Composing it with the homomorphism $P \to M_d(R)$ we get a lifting $\psi_0:A \to M_d(R)$
with $p\psi_0 = \phi$.

We now need to adjust $\psi_0$ to obtain a lifting $\psi$ with $\psi(e_i) = \epsilon_i$ for all $i$.
Let $S = K^n$ with the coordinatewise multiplication, so it is a separable algebra over $K$, 
and let $\sigma:S\to A$ be the map sending the coordinate vectors in $S$ to the idempotents $e_i$.
Let $\eta:S\to M_d(K)$ be the map sending the coordinate vectors to $\epsilon_i$.
Then the condition that $\theta(e_i)=\epsilon_i$ for all $i$ can be rewritten as
$\theta \sigma = \eta$. Similarly $\phi\sigma = \eta$, and we need to find $\psi$ with $\psi\sigma = \eta$.

For $s\in S$ define $d(s) = \psi_0(\sigma(s)) - \eta(s)$. This defines a map from $S$ to $M_d(R)$,
and it has image contained in $M_d(I)$ since $pd = \phi\sigma - \eta = 0$.
Now the bimodule action of $A$ on $M_d(I)$ gives a bimodule action of $S$,
in which, by the discussion above, the left or right action of $s\in S$ can be given
by left or right multiplication by any of $\eta(s)$ or $\phi(\sigma(s))$ or $\psi_0(\sigma(s))$.
Now it is easy to see that $d:S\to M_d(I)$ is a derivation, so, since $S$ is separable,
it is an inner derivation, so there is $\gamma \in M_d(I)$ with 
$d(s) = \eta(s)\gamma - \gamma \eta(s)$ for all $s\in S$.
Letting $g = 1+ \gamma\in \Gl_d(R)$, it follows that
$g \psi_0(\sigma(s)) g^{-1} = \eta(s)$.
Then defining $\psi(a) = g\psi_0(a) g^{-1}$ we obtain the required lift of $\phi$.
\end{proof}

\section{Geometric applications of the projective quotient algebra}

%
%

\subsection*{Desingularizations of quiver Grassmannians for representation-finite algebras}
This is a straightforward generalization of \cite[\S 7]{CIFR}. 
Let $A$ be a finite-dimensional, representation-finite, basic associative $K$-algebra and $M$ a finite-dimensional $A$-module and let $\dd$ be a dimension vector. 
As in the introduction, the irreducible components of $\GrMd$ are of the form 
$\overline{\mcS_{[N_i]}}$ for some $A$-modules $N_1,\dots,N_\ell$.
Theorem \ref{GrassDes} is a direct consequence of the following theorem (compare \cite[Theorem 7.5]{CIFR} for $A$ hereditary).  
 
\begin{thm} Let $N=N_i$ for some $i\in\{1, \ldots , \ell\}$ and let $(\dd,\rr) = \Dim \cp (N)$.
\begin{itemize}
\item[(1)]
$\GrcMd$ is smooth (not just as a variety, but also with its natural scheme structure) and $\overline{\mcS_{[c(N)]}}$ is a connected component. 
\item[(2)] 
The map $\pi\colon \GrcMd\to \GrMd$ induced by multiplication by $e$ is projective
and $\pi^{-1}(U) \cong \Gr_B\binom{\cp (M)/\cp (U)}{\Dim \cp (N)-\Dim \cp (U)}$ for
$U\in \GrMd$.
\item[(3)] 
The image of $\pi $ is closed and contains $\overline{\SN}$. Over an open subset of $\SN$ the map is an isomorphism.  

\end{itemize}
\end{thm}

\begin{proof} 
(1)  This follows from  \cite[Proposition 7.1]{CIFR} since $B$ has global
dimension $\le 2$ and $c(M)$ has injective and projective dimension $\le 1$ and
$\Ext^1_B(c(M),c(M))=0$.

To show that $\overline{\SwN}$ is a connected component, it is enough to see that 
it contains an open subset. Since $\SN$ contains by assumption an open subset, we just need to see $\pi^{-1}(\SN)=\SwN$. So, take $U\subset \cp (M)$ such that $e(U)\cong N$. By Lemma \ref{IntExt} we conclude $\cp (N)$ is a subquotient of $U$ and since they have the same dimension vector they are isomorphic. On the other hand, any point in $\SwN$ maps clearly to $\SN$.   

(2) 
The map $\pi$ is projective since $\GrcMd$ is projective. 
We claim that we have  
\[
\pi^{-1}(U)= \{ F\subset \cp (M) \mid \Dim F =\wdd, \cp (U) \subset F \}. 
\]
It is enough to show that for $F\subset \cp (M)$ with $\Dim F =\wdd $ one has the following equivalence: 
\[ 
\eprime (F) =U \Leftrightarrow \cp (U) \subset F.
\]
So, given $F\subset \cp (M)$ with $\Dim F =\wdd$ and $\eprime (F) =U$. 
Then, we know that $F$ is in $\HKmod$ and stable because it is the submodule of the stable module $\cp (M)$ and by Lemma~\ref{stable}, we have 
\[ 
\cp (U) =\cp \eprime (F) \subset F. 
\]
On the other hand, given $F\subset \cp (M)$.  
with $\Dim F =\wdd$ and $\cp (U) \subset F$. Since $\eprime $ is exact, we get $U\subset \eprime (F)$. But we also have for every simple $A$-module $S_i$ 
\[ 
\begin{aligned}
(\Dim U)_i & = \dim \Hom (P_i, U) =\dim \Hom (P_i, N) = \\
\dim \cp (N)(P_i\to 0) &=\dim F(P_i\to 0) = (\Dim \eprime (F))_i 
\end{aligned}
\] 
and therefore $U=\eprime (F)$.
 
(3) Since $\pi $ is projective, its image is closed. For every $U\in \SN$ we have $\cp (U) \in \SwN$ and $\pi (c(U)) =U$, therefore $\overline{\SN}$ is contained in the image. By (2), we get that the morphism restricts to a bijection 
$\pi^{-1}(\SN)=\SwN \to \SN$. To see that it is an isomorphism over an open set we just need to see that it is an isomorphism on tangent spaces over an open set. Now, the map $\pi$ has a scheme-theoretic version $\pi^\prime$ -straightforward to define on commutative $K$-algebra valued points- such that the map on underlying reduced schemes coincides with $\pi$. Let $U\in \GrcMd$ (be a $K$-valued point), then the tangent map of $\pi^\prime$ at $U$ is given by $\Hom_B(U, c(M)/U) \xrightarrow{e} \Hom_A( e(U), M/e(U))$. Now, we prove for $U\in \SwN$, the map on tangent space is injective as follows: Since $U$ is bistable, it equals $ce(U)$. The canonical epimorphism $c(M)/U=c(M)/ce(U) \to c( M/e(U))$ has a kernel $X$ with $e(X)=0$, applying $\Hom_B(U,-)$ gives an exact sequence 
\[ 
0\to \Hom_B(U,X)\to \Hom_B(U,c(M)/U )\to\Hom_B(U,c( M/e(U))\cong \Hom_A(e(U),M/e(U))
\]
Since $U$ is is bistable, we get $\Hom_B(U,X)=0$ and the second map equals the tangent map.

Since $\GrcMd$ is reduced, the tangent map actually factors through the tangent space of the variety $\GrMd$ at $e(U)$. For $U\in \SwN$ we have 
\[ 
\dim \Hom_B(U, c(M)/U) =\dim \Hom_A (U, M) - \dim \Hom_A(U,U) =\dim \SN.
\]
Since the variety $\GrMd$ has to have some smooth points in $\SN$, we conclude that for these the map on tangent spaces is an isomorphism. 
\end{proof}


%
%

\subsection*{Orbit closures as quotients}
Let $\RdrB$ be the representation space of all $(\dd, \rr)$-dimensional $B$-modules. 
The functor $e$ induces a map
$\RdrB \to \RdA$, also denoted $e$.
We observe the following. 

\begin{lem} 
\label{ImRes} 
Assume $\RdrB\neq \emptyset $ and $N\in \RdA$. Then the following are equivalent. 
\begin{itemize}
\item[(1)] $N\in \Bild \eprime $
\item[(2)] $\Dim \cp (N) \leq (\dd, \rr )$ (pointwise at every idempotent for $B$)
\end{itemize}
In particular, $\Bild \eprime $ is a closed subset of $\RdA$.
\end{lem}

\begin{proof}
Assume $N= \eprime (F)$ for some $F\in \RdrB$. 
By Lemma~\ref{IntExt} we get $\Dim \cp (N) \leq \Dim F= (\dd , \rr)$. 
Now assume $\Dim \cp (N) \leq (\dd , \rr)$. 
Because the algebra $B$ is basic and finite-dimensional, there is a unique semi-simple $B$-module $S_1$ of dimension $\Dim \cp (N)$ and a unique semi-simple $B$-module $S_2$ of dimension 
$(\dd, \rr )$. Since $\Dim \cp (N) \leq (\dd, \rr)$ one has $S_1$ is a submodule of $S_2$, set 
$S= S_2/S_1$. We have $\eprime (S)=0$ and the dimension of $S$ is $(\dd, \rr)-\Dim \cp (N)$. 
Set $F= \cp (N) \oplus S$, this gives a point in $\RdrB$ with $\eprime (F) =N$, so $N \in \Bild \eprime$. 
Since we are in the representation-finite case, we just need to see the following: 
Let 
\[0\to N_1\xrightarrow{i}M \xrightarrow{p} N_2\to 0\] 
be a short exact sequence of $A$-modules (with $\Dim M=\dd)$ and $\Dim \cp (M) \leq (\dd, \rr)$, 
then $\Dim \cp (N_1\oplus N_2) \leq (\dd, \rr)$. 
We look at the (not exact) sequence 
\[ 
\cp (N_1)\xrightarrow{\cp (i)} \cp (M) \xrightarrow{\cp (p)} \cp (N_2)
\]
since $\cp $ is an additive functor we have 
$\Bild \cp (i) \subset \Ker \cp (p)$ and since intermediate extensions preserve mono- and epimorphisms we have that $\cp(i)$ is a monomorphism and $\cp (p) $ is an epimorphism. 
This implies $\Dim \cp (N_1\oplus N_2) \leq \Dim \cp (M)$. 
\end{proof}

Let $M$ be a ($K$-valued) point in $\RdA$. We consider the closure of its $\Gd$-orbit $\overline{\mcO_M} \subset \RdA$ (with the reduced scheme structure). 

\begin{lem} If $\Dim \cp (M)= (\dd , \rr )$,
then the (set-theoretic) image of $\eprime \colon \RdrB \to \RdA$ is $\overline{\mcO_M}$. 
\end{lem}

\begin{proof} 
The image is closed and $\Gd$-invariant in $\RdA $ (see the previous lemma). 
Since $\eprime \cp(M) =M$ we conclude that 
$\overline{\mcO_M}$ is a subset in the image of $\eprime $. 
Now, given $N$ in the image of $\eprime $, let us say $N=\eprime (F)$. 
Then, by Lemma~\ref{ImRes} we get $\Dim \cp (N) \leq \Dim \cp (M)$. 
We conclude for $P\to X$ indecomposable in $\mcH $ 
\[
\begin{aligned}
\dim \cp (N) (P\to X) & =\dim \Hom_A (P,N) - \dim \Hom_A (X,N) \\
                          &\leq \dim \Hom_A(P,M) - \dim \Hom_A (X,M) \\
                          &= \dim \cp (M) (P\to X).
\end{aligned}
\]
But since $\Dim N= \Dim M$ we have $\dim \Hom_A (P,N)=\dim \Hom_A (P,M)$, 
so $\dim \Hom_A (X,N)\geq \dim \Hom_A (X,M)$,
and this implies by \cite{Zw5} that $N\in \overline{\mcO_M}$.
\end{proof}

We now have our generalization of \cite[Theorem 1.2]{CFR}.

\begin{thm}
\label{CFR-Thm} 
Assume that $K$ has characteristic zero. 
If $M$ is an $A$-module and $(\dd, \rr ) =\Dim \cp (M)$,
then $\overline{\OM}$ is isomorphic to the affine quotient variety
$\RdrB\dbslash\Glr$, and also to $\overline{\mcO_{c(M)}} \dbslash \Glr$.
\end{thm}

\begin{proof}
The first isomorphism holds since the map
$\eprime \colon \RdrB \dbslash \Glr \to \RdA$ is a closed immersion
by Lemma~\ref{FFT}, and it has set-theoretic image 
$\overline{\mcO_M}$ by the previous lemma. 
The second isomorphism is straightforward (see \cite[Theorem 3.3]{CFR}).
\end{proof}

%

\subsection*{Desingularization of orbit closures}

Let $K$ be an algebraically closed field of characteristic zero. 
Let $Q_B$ be a quiver as in \S\ref{deframedquiver} and let $Q^\infty$ be the deframed 
quiver constructed from a given dimension vector $(\dd, \rr)$. 
We consider the stability notion on $\RrQinfty $ given by the vector $\theta \in \R^{m-n+1}$ defined as $\theta_i = -1 $ if $n+1\leq i\leq m$ and $\theta_{\infty} = \sum_{i=n+1}^m d_i$. 
Thus a representation $M$ in $\RrQinfty $ is $\theta$-stable if and only if for any non-zero proper subrepresentation $N$ of $M$ one has $\theta \cdot \Dim N > 0$. One can also consider $\theta$-semi-stable representations $M$ defined by the condition $\theta \cdot \Dim N \geq 0$. But for this particular $\theta$ these two notions are equivalent. 
By geometric invariant theory (compare \cite{Ki}) there is a projective morphism of varieties 
\[ 
\RrQinfty^{s} / \Gl_{(1, \rr ) } \to \RrQinfty \dbslash \Gl_{(1, \rr )} 
\]
where $()^{s}$ denotes the subset of $\theta$-stable points and the quotient on the left hand side is a geometric quotient. 

Now recall, that we have an $\Glr$-equivariant isomorphism of varieties  
$\RdrQB \to \RrQinfty$. It is easy to see that under this isomorphism the $\theta$-stable points correspond to the points of $\RdrQB$ which are stable $KQ_B$-modules with respect to the recollement given by the idempotent $e=\sum_{i=1}^n e_i$.  
Thus the projective morphism becomes 
\[ 
\RdrQB^{s} / \Glr \to \RdrQB \dbslash \Glr 
\]
where $()^{s}$ denotes the stable points.

Now, let $B$ be a finitely generated algebra and let $e_1, \ldots , e_n, e_{n+1}, \ldots , e_m$
($n\le m$) be a complete set of orthogonal idempotents in $B$.
By writing $B$ as a quotient of a path algebra $KQ_B$ one obtains a projective morphism 
\[ 
\RdrB^{s} / \Glr \to \RdrB \dbslash \Glr.
\]
The left hand side is a geometric quotient---to see that it is a good quotient, use e.g.\ \cite[Theorem 7.1.4]{BB} and the other properties check directly. 

Now let $(\dd, \rr) =\Dim c(M)$. Since $\overline{\mcO_{c(M)}}\cap \RdrB^{s}$ is closed in $\RdrB^{s}$, the geometric quotient by $\Glr$ exists and by composing with the closed immersion and the isomorphism from the previous theorem one gets a projective map  
\[ 
\varphi\colon 
(\overline{\mcO_{c(M)}}\cap \RdrB^{s}) / \Glr \to \RdrB^{s} / \Glr 
\to \RdrB \dbslash \Glr \cong \overline{\mcO_M}.  
\]

\begin{thm}\label{OrbDes}$\varphi$ is a desingularization. 
\end{thm}

\begin{proof}
We claim that $\RdrB^{s}/\Glr$ is smooth. Since the global dimension of $B$ is at most $2$, we can use Lemma~\ref{Ext2} to see that the second self-extension group of every stable point vanishes. By Lemma~\ref{Ext2smooth}, we get that this is a smooth point in $\RdrB$, therefore the open subvariety 
 $\RdrB^{s}$ is smooth. It follows that $\RdrB^{s}/ \Glr$ is smooth because $\RdrB^{s}\to \RdrB^{s}/ \Glr$ is a principal bundle. Now, $(\overline{\mcO_{c(M)}}\cap \RdrB^{s}) / \Glr$ is a connected component of it, since $\overline{\mcO_{c(M)}}$ is an irreducible component of $\RdrB$ because $c(M)$ is rigid. 

We need to see that $\varphi$ is bijective over $\mcO_M$. The preimage of $\mcO_M$ under the map $\RdrB \to \RdrB\dbslash\Glr \cong \overline{\mcO_M}$ is the $\Gl_{(\dd, \rr) }$-orbit of $\cp (M)$. To see this, take a a point $F$ in the fibre over $M$. Then $\cp (M)$ is a subquotient of $F$ by 
Lemma~\ref{IntExt} which has the same dimension as $F$, therefore they are isomorphic. The restriction $\mcO_{\cp (M)} \to \mcO_{M}$ is a principal $\Glr$-bundle. By \cite[Theorem 5.25]{Bo} the geometric quotient exists and the map factors over an induced isomorphism 
$ \mcO_{\cp (M) } / \Glr \to \mcO_M $. But this coincides with the restriction of our map $\RdrB^{s}/\Glr \to \RdrB\dbslash\Glr \cong \overline{\mcO_M}$ to the preimage of $\mcO_M$. 
\end{proof}

For orbit closures of Dynkin quivers there already exists a resolution of singularities constructed by Reineke using the directedness of the Auslander-Reiten quiver of Dynkin quiver, see \cite{R}.

\section{Examples of projective quotient algebras}

In this section, we choose representation-finite algebras $A$ and describe the associated projective quotient algebras $B$.

\subsection{Hereditary algebras of finite representation type.} 
In this case $A=KQ$ with $Q$ Dynkin quiver (i.e. the underlying graph is Dynkin of type A,D or E).  
The homotopy category $\mcH$ is equivalent to the full subcategory of $\mcQ$ given by the objects 
$f\colon P\to X$ such that $X$ has no non-zero projective summand and so it is equivalent to the category $\mcH_Q$ of \cite{CIFR} and \cite{CFR}. 
Our algebra $B$ coincides with the algebra $B_Q$ from loc.\ cit.\ where they calculate its Ext-quiver and the relations from extensions between simple modules. 
Our results for this case are as loc.\ cit. The desingularization of orbit closures is not explicitly considered there, but other desingularizations are already known, see \cite{R}, and the
Hernandez-Leclerc construction \cite{HL} relates orbit closures to Nakajima quiver varieties, and in this context moduli spaces have previously been used to obtain desingularizations, 
cf.\ \cite[Theorem 3.2]{Sch}.

\subsection{Self-injective algebras of finite representation type.}
In this case the functor $\Ker$ from $\mcH$ to $A$-$\modu$ sending an object 
$f\colon P\to X$ to $\Ker (f)$ is an equivalence of categories. So, $B$ is isomorphic to the Auslander algebra $\Gamma $ of $A$. Moreover, we have a commuting diagram of functors 
\[
\xymatrix{
\text{$B$-$\Modu$} \ar[r]^{e}& \text{$A$-$\Modu$} \ar@{=}[d]\\
\text{$\Gamma$-$\Modu$} \ar[u]^{\widehat{\Ker} }\ar[r]^{\ep}& \text{$A$-$\Modu$} 
}
\] 
where $\ep=\sum_{i=1}^n e_{[P_i]}$ is the idempotent in $\Gamma =\End_A(E )^{op}$ corresponding to the summand of $E$ corresponding to indecomposable projective modules.

\subsection{The truncated polynomial ring.}
Let $A=K[X]/(X^n)$. This algebra is representation-finite and self-injective. 
The indecomposables in $\mcQ$ are the objects 
$U_r:=(A\to K[X]/(X^r))$ with $0\leq r \leq n$ and the 
indecomposables in $\mcH$ are those with 
$r<n$. (But note that $\mcH$ is not equivalent to the full subcategory of $\mcQ$ containing these indecomposables.)
The Auslander algebra of $A$ (and therefore also $B$) has a unique structure as quasi-hereditary algebra. Explicitly, we can describe $B$ with the quiver 
\[
\xymatrix{
n-1 \ar@<1ex>[rr]^{p_{n-1}} & & n-2 \ar@<1ex>[ll]^{j_{n-1}} \ar@<1ex>[rr]^{p_{n-2}} &&\ar@<1ex>[ll]^{j_{n-2}}  \cdots\cdots  \ar@<1ex>[rr]^{p_3}&& 2 \ar@<1ex>[rr]^{p_{2}} \ar@<1ex>[ll]^{j_{3}} && 1 \ar@<1ex>[rr]^{p_{1}} \ar@<1ex>[ll]^{j_{1}} && [0] \ar@<1ex>[ll]^{j_{1}}
}
\]
and relations 
$p_rj_r = j_{r-1}p_{r-1}, \quad 0 < r-1 < n-1, \text{ and } p_{n-1} j_{n-1}=0$. The brackets $[-]$  indicate the idempotent $e$ such that $eBe=A$. The stable modules are given by the modules $F$ with 
$\Hom (U_r, F)=0$ for $1\leq r\leq n$, this means in $F$ all maps $p_*$ have to be monomorphisms. The costable modules are given by modules $H$ with $\Hom (H, U_r)=0$ for $1\leq r\leq n$, this means in $H$ all maps $j_*$ have to be epimorphisms. 
The stable modules coincide with the $\Delta$-filtered modules and the costable modules with the $\nabla$-filtered modules for the unique quasi-hereditary structure (see next example or e.g. \cite{BHRR}). In particular, 
the tilting module $C$ coincides with the characteristic tilting module. 

Our geometric construction of the orbit closures as affine quotient varieties coincides with the one from Kraft and Procesi in \cite[\S 3.3]{KP}. Their variety $Z$ equals our $\RdrB$ and the union of the stable and the costable locus is contained in their smooth variety $Z^0$. 

\subsection{The nilpotent oriented cycle.}

Let $Q$ be the quiver with vertices $\{1, \ldots , N\}$ identified with their residue classes in the additive group $\Z/N\Z$. 
For each vertex $i$, we have one arrow $x_{i}\colon i\to i+1$. Let $I$ be the ideal given by all path of length $n$ in $Q$. Then, the algebra $A=KQ/I$ is a  representation-finite, self-injective Nakayama algebra, for $N=1$ we reobtain the previous example. We denote by $E_{i}[r]$ the indecomposable $A$-module with top $S_i$ of dimension $r$, $i\in \Z/N\Z, \; r\in \{1, \ldots , n\}$. If we set $E_i[0]:=0$, the indecomposable objects in $\mcH$ are 
$U_{i,r}:= (E_i[n]\to E_i[r])$, $i\in \Z/N\Z$, $0\leq r\leq n-1$. The algebra $B$ can be described by the quiver with vertices $(i,r)$, $i\in \Z/N\Z$, $0\leq r\leq n-1$ and arrows and relations see below (for $N=3, n=4$ with identification of the left and right boundary) 
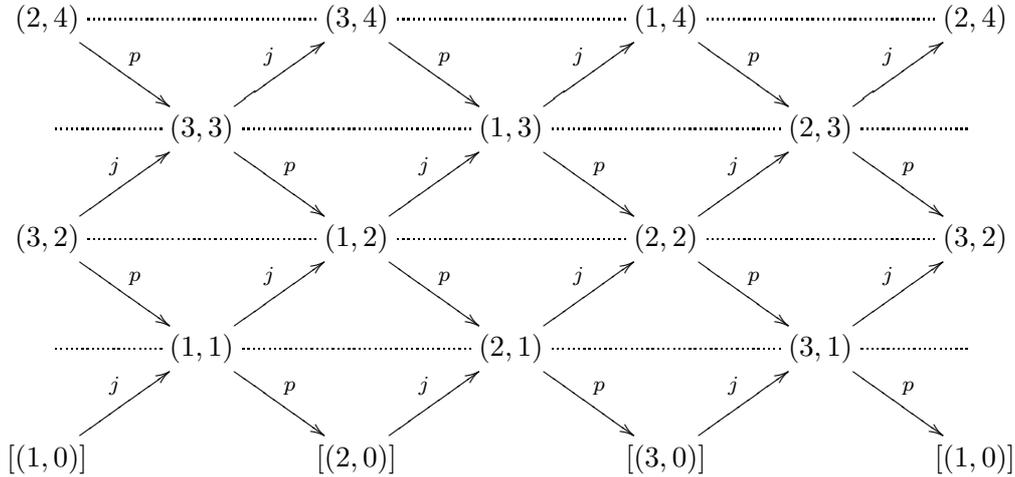
\[
\xymatrix{
(2,4)\ar[dr]^p && (3,4)\ar[dr]^p\ar@{.}[ll]&& (1,4) \ar[dr]^p\ar@{.}[ll]&& (2,4)\ar@{.}[ll] \\
& (3,3)\ar[dr]^p \ar[ur]^j\ar@{.}[l] && (1,3)\ar[ur]^j \ar[dr]^p\ar@{.}[ll] && (2,3)\ar[ur]^j \ar[dr]^p\ar@{.}[ll] & \ar@{.}[l]\\
(3,2)\ar[ur]^j\ar[dr]^p && (1,2)\ar[ur]^j\ar[dr]^p\ar@{.}[ll] && (2,2)\ar[ur]^j\ar[dr]^p\ar@{.}[ll]&& (3,2) \ar@{.}[ll] \\
& (1,1)\ar[ur]^j\ar[dr]^p\ar@{.}[l]&& (2,1)\ar[ur]^j\ar[dr]^p\ar@{.}[ll] && (3,1)\ar[ur]^j\ar[dr]^p\ar@{.}[ll] &\ar@{.}[l] \\
[(1,0)]\ar[ur]^j && [(2,0)] \ar[ur]^j&& [(3,0)]\ar[ur]^j && [(1,0)]
}
\]
the $[-]$ indicates the idempotent $e$ with $eBe=A$. The stable modules are given by the condition that all maps $p$ are monomorphisms and the costable modules by the condition that all maps $j$ are epimorphisms. There is a quasi-hereditary structure on $B$ such that the stable modules are 
the $\Delta$-filtered and the costables are the $\nabla$-filtered, this can e.g. be obtained by the following ordering of the vertices in the quiver above: Take any total ordering refining $(i,r)<(j,t)$ for all $r<t$ and use the convention 
$\Delta (\la ) =P(\la ) /\sum_{\mu >\la } \Bild \left( P(\mu ) \to P(\la )\right)$, 
$\nabla (\la ) = \bigcap_{\mu >\la } \Ker \left(Q(\la )\to Q(\mu )\right)$ where $Q(\la )$ is the injective hull of the simple supported at vertex $\la $.  

In this case the orbit closures are normal, desingularizations are known (by adapting the construction of \cite{R}) and they are dense subvarieties of affine Schubert varieties of type $A$, see \cite[Proposition 1.1 and 6.2]{Sc}.

\subsection{The commuting square.} Let $A$ be given by the following quiver with 
relation 
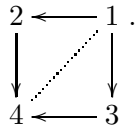
\[ 
\xymatrix{
2 \ar[d] & 1 \ar@{.}[dl]\ar[l]\ar[d] \\
4 & 3\ar[l]
}.
\]
This algebra is of finite representation type, tilted of type $D_4$. But the homotopy category $\mcH$ is not a full subcategory of $\mcQ$. We write $[U]$ for the indecomposable object $P_U\to U$ with $U$ indecomposable non-projective and $[i]$ for the indecomposable $P_i\to 0$ in $\mcH$. 
The algebra $B$ is given by the following quiver with relations. 
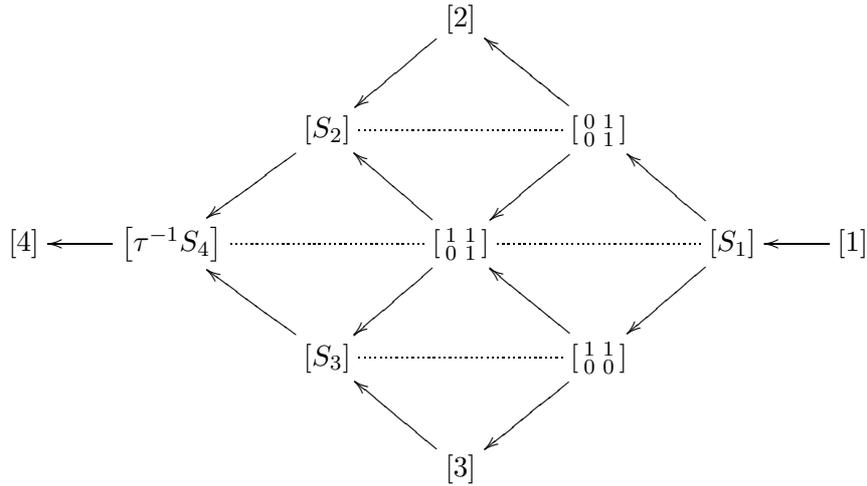
\[
\xymatrix{
&&& [2]\ar[dl] &&&  \\
&& \left[ S_2 \right] \ar[dl] && 
 \left[ {\begin{smallmatrix} 0 & 1 \\ 0 & 1 \end{smallmatrix}} \right]\ar[ul] \ar[dl]\ar@{.}[ll] &&  \\
[4] & \left[ \tau^{-1} S_4 \right] \ar[l] &&  \left[ {\begin{smallmatrix} 1&1\\ 0&1 \end{smallmatrix}} \right] \ar[ul]\ar[dl]\ar@{.}[ll]&& \left[ S_1\right] \ar[ul]\ar[dl]\ar@{.}[ll]& [1] \ar[l]\\
&& \left[ S_3\right] \ar[ul]&& \left[ {\begin{smallmatrix} 1&1\\ 0&0 \end{smallmatrix}} \right]\ar[ul]\ar[dl] \ar@{.}[ll] && \\
&&& [3] \ar[ul]&&& }
\]
Here $e= e_{[1]}+e_{[2]}+e_{[3]}+ e_{[4]}$ is the idempotent such that $eBe=A$. The intermediate extensions of the indecomposables are given as follows: $\cp (S_i) =S_{[i]}$ for $1\leq i\leq 4$, and 
\[
\cp \left(\begin{smallmatrix} 1 & 0 \\ 1 & 0 \end{smallmatrix} \right)
=\begin{smallmatrix}
&&& 1 &&&\\
&& 1 && 0 && \\
1 & 1 && 0 && 0 & 0 \\ 
&& 0 && 0 && \\
&&& 0 &&&
\end{smallmatrix}, \quad 
\cp \left(\begin{smallmatrix} 0 & 1 \\ 0 & 1 \end{smallmatrix} \right)
=\begin{smallmatrix}
&&& 0 &&&\\
&& 0 && 0 && \\
0 & 0 && 0 && 1 & 1 \\ 
&& 0 && 1 && \\
&&& 1 &&&
\end{smallmatrix}, \quad 
\cp \left(\begin{smallmatrix} 1 & 0 \\ 1 & 1 \end{smallmatrix} \right)
=\begin{smallmatrix}
&&& 1 &&&\\
&& 1 && 0 && \\
1 & 1 && 0 && 0 & 0 \\ 
&& 1 && 0 && \\
&&& 1 &&&
\end{smallmatrix}, \quad 
\cp \left(\begin{smallmatrix} 1 & 1 \\ 1 & 1 \end{smallmatrix} \right)
=\begin{smallmatrix}
&&& 1 &&&\\
&& 1 && 1 && \\
1 & 1 && 1 && 1 & 1 \\ 
&& 1 && 1 && \\
&&& 1 &&&
\end{smallmatrix} 
\]
and similarly for $\begin{smallmatrix} 1 & 1 \\ 0 & 0 \end{smallmatrix} , \begin{smallmatrix} 0 & 0 \\ 1 & 1 \end{smallmatrix}, \begin{smallmatrix} 1 & 0 \\ 1 & 0 \end{smallmatrix}$ and $\begin{smallmatrix} 1 & 1 \\ 0 & 1 \end{smallmatrix} $.  
Another desingularization for the quiver Grassmannians (for algebras  iterated tilted of Dynkin type) has been found in \cite{KS2} where they discuss the above example. Their construction uses the repetitive algebra of a Dynkin quiver of type $D_4$.

\addcontentsline{toc}{section}{\textbf{References} \hfill}

\bibliographystyle{alphadin}
\bibliography{qgocrepfinite}

\end{document}